\documentclass[10pt,a4paper]{article}
\usepackage[top=3cm, bottom=3cm, left=3cm, right=3cm]{geometry}
\usepackage[latin1]{inputenc}
\usepackage{amsmath}
\usepackage{amsthm}
\usepackage{amsfonts}
\usepackage{amssymb}
\usepackage{graphics}
\usepackage[hang,flushmargin]{footmisc} 

\usepackage{amsmath,amsthm,amssymb,graphicx, multicol, array}
\usepackage{enumerate}
\usepackage{enumitem}
\usepackage{xcolor}
\usepackage{currfile}

\usepackage[bookmarks,colorlinks,breaklinks]{hyperref}

\hypersetup{linkcolor=blue,citecolor=blue,filecolor=blue,urlcolor=blue} 

\usepackage{tabto}

\usepackage{tikz}
\usepackage{pgfplots}

\DeclareMathOperator{\ord}{ord}

\DeclareMathOperator{\tmod}{mod}
\DeclareMathOperator{\tdiv}{div}

\newtheorem{thm}{Theorem}[section]
\newtheorem{lem}{Lemma}[section]
\newtheorem{conj}{Conjecture}[section]
\newtheorem{exa}{Example}[section]

\newtheorem{dfn}{Definition}[section]
\newtheorem{exe}{Exercise}[section]

\newcommand{\N}{\mathbb{N}}
\newcommand{\Z}{\mathbb{Z}}

\newcommand{\C}{\mathbb{C}}
\newcommand{\F}{\mathbb{F}}

\usepackage{tocloft}

\usepackage{fancyhdr}
\title{Configurations Of Consecutive Primitive Roots}
\author{N. A. Carella}

\begin{document}
\thispagestyle{empty}
\date{}
\maketitle

\vskip .25 in 
\textbf{\textit{Abstract}:} Let \(p \geq 2\) be a large prime, and let $k \ll \log p $ be a small integer. This note proves the existence of various configurations of $(k+1)$-tuples of consecutive and quasi consecutive primitive roots $n+a_0, n+a_1, n+a_2, \ldots, n+a_k$ in the finite field $\F_p$, where $a_0,a_1, \ldots, a_k$ is a fixed $(k+1)$-tuples of distinct integers.  \let\thefootnote\relax\footnote{
\date{\today}
}

\vskip .25 in 
\tableofcontents
\newpage
\section{Introduction}  \label{s8800}
Let \(p \geq 2\) be a large prime, and let $\F_p$ be a finite field. The order $\ord_p \alpha= d$ of an element $\alpha \in \F_p$ is the smallest divisor $d\mid p-1$ for which $\alpha^d\equiv 1 \bmod p$. An element of maximal order $\ord_p(\alpha)=p-1$ is called a primitive root. This note is concerned with the configurations of subsets of primitive roots in finite fields. A configuration deals with the existence of $(k+1)$-tuples of quasi consecutive primitive roots 
\begin{equation} \label{eq8800.030}
n +a_0, \quad n+a_1,\quad n+a_2,\quad \ldots,\quad n+a_k, 
\end{equation}
where $a_0, a_1, \ldots, a_k$ is a fixed $(k+1)$-tuples of distinct integers, in a finite field $\F_p$, or in large subsets $\mathcal{A} \subset \F_p$. The corresponding counting functions have the forms 
\begin{equation} \label{eq8800.035}
\sum_{ n\in \F_p}  \Psi \left(n+a_0\right)\Psi \left(n+a_1\right )\cdots \Psi \left(n+a_k\right)f\left(n+a_0\right)f\left(n+a_1\right)\cdots f \left(n+a_k\right),
\end{equation}
and
\begin{equation} \label{eq8800.037}
\sum_{ n\in \mathcal{A}}  \Psi \left(n+a_0\right)\Psi \left(n+a_1\right )\cdots \Psi \left(n+a_k\right)f\left(n+a_0\right)f\left(n+a_1\right)\cdots f \left(n+a_k\right),
\end{equation}
respectively, where $\Psi :\mathbb{N} \longrightarrow \{0,1\}$ is the characteristic function of primitive roots modulo $p$, see Section \ref{s333}, and $f: \mathbb{N}\longrightarrow \Z$ is an arithmetic function. The function $f$ restricts the sequence of $(k+1)$-tuples of quasi consecutive primitive roots to certain subsequence of integers. There are many possible classes of clusters and constellations of primitive roots generated by the different classes of $(k+1)$-tuples. The precise results for some of the various restricted $(k+1)$-tuples of configurations of quasi consecutive primitive roots are detailed below.

\subsection{Consecutive Primitive Roots}
The earliest works on consecutive primitive roots seems to be that in \cite{CL56} or before. The author proved a general result for the existence of consecutive primitive roots. The proof is based on the divisors dependent characteristic function for primitive roots, see Lemma \ref{lem333.02}. Later, a qualitative result for the existence of some consecutive primitive roots was proved in \cite{VE68}. A quantitative and weaker result for two consecutive primitive roots is proved in \cite{SM75}, the same result, but emphasizing the numerical aspects, is also proved in \cite{CS85}. More recently, some partial result but no proof for $k$-consecutive primitive roots appears in \cite{TT13}.

\begin{thm}  \label{thm8800.040}  Let \(p\geq 2\) be a large prime, and let $k\ll \log p $ be an integer. Then, the finite field $\F_p$ contains $(k+1)$-tuples of consecutive primitive roots. Furthermore, the number of $(k+1)$-tuples has the asymptotic formula
\begin{equation} \label{eq8800.045}
N(k,p)=\left (\frac{\varphi(p-1)}{p-1}\right )^{k+1}p+O(p^{1-\varepsilon}),
\end{equation}
where $\varepsilon>0$ is an arbitrary small number.
\end{thm} 
The complete proof for this case is given in Section \ref{s9090}.
\begin{thm}  \label{thm8800.050}  Let \(p\geq 2\) be a large prime, and let $k\ll \log p $ be an integer. Then, any large subset of elements $\mathcal{A} \subset \F_p$ of cardinality $p^{1-\varepsilon/2}\ll \# \mathcal{A}$ contains $(k+1)$-tuples of consecutive primitive roots. Furthermore, the number of $(k+1)$-tuples has the asymptotic formula
\begin{equation} \label{eq8800.055}
N(k,p, \mathcal{A})=\left (\frac{\varphi(p-1)}{p-1}\right )^{k+1}\# \mathcal{A}+O(p^{1-\varepsilon}),
\end{equation}
where $\varepsilon>0$ is an arbitrary small number.
\end{thm} 

The average length of $(k+1)$-tuples is $k \ll \log p/\log \log \log p$. This statistic is dependent on the primes decomposition of the average totient $p-1$. Asymptotically, highly composite totients $p-1$ have slightly shorter lengths $k \ll \log p/ \log \log p$. The Fermat and Germain totients have the longest lengths, namely, $k \ll \log p$, the details appears in Lemma \ref{lem9799.156}. The distribution of $(k+1)$-tuples of consecutive primitive roots is a very interesting research problem. The numerical data is not adequate to make any strong heuristic, but it suggests that the $(k+1)$-tuples of consecutive primitive roots are not uniformly distributed.

\subsection{Consecutive Squarefree Primitive Roots}
The result for a single squarefree primitive root $n$ in a finite field $\F_p$, which is a special case of Theorem \ref{thm8800.190}, is proved in Theorem \ref{thm9010.040}. A result for two consecutive squarefree primitive roots $n$ and $n+1$ in a finite field $\F_p$ is given in Theorem \ref{thm9040.040} and a result for three consecutive squarefree primitive roots $n$, $n+1$ and $n+2$ is given in Theorem \ref{thm9030.050}. The next case for four squarefree primitive roots $n$, $n+1$, $n+2$ and $n+3$ is not feasible, see \eqref{eq297.100}. However, there are other sequences of integers that support long strings of quasi consecutive squarefree primitive roots.

\begin{thm}  \label{thm8800.080}  Let \(p\geq 2\) be a large prime, and let $k\ll \log p$ be an integer. For any admissible $(k+1)$-tuples $a_0<a_1< \cdots<a_k$, the finite field $\F_p$ contains $(k+1)$-tuples of consecutive squarefree primitive roots
\begin{equation} \label{eq8800.200}
n+a_0, \quad n+a_1,\quad n+a_2,\quad \ldots,\quad n+a_k.
\end{equation}
Furthermore, the number of $(k+1)$-tuples has the asymptotic formula
\begin{equation} \label{eq8800.045}
N(k,p)=\prod_{q\geq 2}\left ( 1-\frac{\omega(q)}{q^2}\right )\left (\frac{\varphi(p-1)}{p-1}\right )^{k+1}p+O(p^{1-\varepsilon}),
\end{equation}
where $\varepsilon>0$ is an arbitrary small number.
\end{thm} 
The complete proof for this case is given in Section \ref{s9050}.
\subsection{Consecutive $s$-Power Free Primitive Roots}
 Let $s \geq 2$ be a small integer. A primitive root $n \in \F_p$ is $s$-power free if and only if it is not divisible by an $s$-power, exempli gratia, $r^s\nmid n$ for all prime $r\geq2$. This idea generalizes the idea of squarefree primitive roots.
\begin{thm}  \label{thm8800.190}  Let \(p\geq 2\) be a large prime, and let $s\geq 2 $ be a small integer. Then, the finite field $\F_p$ contains $s$-power free primitive roots. Furthermore, the number of such elements has the asymptotic formula
\begin{equation} \label{eq8800.145}
N_s(p)=\frac{1}{\zeta(s)}\frac{\varphi(p-1)}{p-1}p+O(p^{1-\varepsilon}),
\end{equation}
where $\zeta(s)$ is the zeta function, and $\varepsilon>0$ is an arbitrary small number.
\end{thm}

\begin{thm}  \label{thm8800.195}  Let \(p\geq 2\) be a large prime, and let $a_0\ne a_1$ and $ s \geq 2$ be small integers. Then, the finite field $\F_p$ contains a pair of consecutive $s$-power free primitive roots $n+a_0$ and $n+a_1$. Furthermore, the number of such pairs has the asymptotic formula
\begin{equation} \label{eq8800.145}
N_s(2,p)=\prod_{q\geq 2}\left (1-\frac{\rho(q)}{q^s} \right ) \left (\frac{\varphi(p-1)}{p-1}\right )^2p+O(p^{1-\varepsilon}),
\end{equation}
where $\rho(s)=1,2$, and $\varepsilon>0$ is an arbitrary small number.
\end{thm}

The complete proofs for these cases are given in Section \ref{s9110}.
\subsection{Consecutive Primitive Roots And Relatively Prime}
The earliest work considered the existence of primitive roots relatively prime to $p-1$. In other words, the case $q=p-1$ was proved in \cite{HM76} using the divisors dependent characteristic function in Lemma \ref{lem333.02}. A generalized version for $q \leq p-1$, using the divisors free characteristic function in Lemma \ref{lem333.03}, is realized in Theorem \ref{thm8800.090}. In addition, for $a \geq 1$, a result for two consecutive primitive roots $n$, $n+a$, and relative prime to $q=q(a)$ is proved in Theorem \ref{thm8800.095}. Both of these results are appear to be new in the literature.
\begin{thm}  \label{thm8800.090}  Let \(p\geq 2\) be a large prime, and let $q< p $ be an integer. Then, the finite field $\F_p$ contains primitive roots relatively prime to $q$. Furthermore, the number of such elements has the asymptotic formula
\begin{equation} \label{eq8800.045}
N_r(p,q)=\frac{\varphi(q)}{q}\frac{\varphi(p-1)}{p-1}p+O(p^{1-\varepsilon}),
\end{equation}
where $\varepsilon>0$ is an arbitrary small number.
\end{thm}

\begin{thm}  \label{thm8800.095}  Let \(p\geq 2\) be a large prime, let $q< p $ be an integer, and let $a\geq 1$ be a fixed integer. Then, the finite field $\F_p$ contains a pair of quasi consecutive primitive roots $n$, $n+a$, relatively prime to $q=q(a)$. Furthermore, the number of such pairs has the asymptotic formula
\begin{equation} \label{eq8800.045}
N_r(2,p,q)=c_2(q,a)\left (\frac{\varphi(q)}{q}\right )^2\frac{\varphi(p-1)}{p-1}p+O(p^{1-\varepsilon}),
\end{equation}
where $c_2(q,a)\geq 0$ is a dependence correction factor, and $\varepsilon>0$ is an arbitrary small number.
\end{thm}

Both parameters $c_2(q,a)\geq 0$ and $q=q(a)$ depend on $a \geq 1$. For instance, for $a=2b+1$ odd, the value $q=q(a)$ must be odd, and $c_2(q,a)>0$, otherwise $c_2(q,a)= 0$ for even $q$. The complete proof for both of these cases are given in Section \ref{s2887}.

\section{Results For Arithmetic Functions } \label{s8533}
Several results for some arithmetic functions required in later sections are recorded here.

\subsection{Prime Divisors Counting Function}\label{s8533A}
Let $p_i\geq 2$ denotes the $i$th prime in increasing order, and let $n \in \N$ be an integer. An integer has a unique prime decomposition $n=p_1^{v_1} \cdot p_2^{v_2} \cdots p_t^{v_t} $, where $v_i\geq1$. 

\begin{dfn}  \label{dfn297.05}{\normalfont The prime divisors counting function $\omega:\mathbb{N} \longrightarrow \N$ is defined by $\omega(n)=t$.
}
\end{dfn}

The number of prime divisors $\omega(n)$ of a random integer $n \in \N$ is a normal random variable with mean $\log \log n$, and standard error $\sqrt{\log \log n}$, as verified below. 

\begin{thm} \label{thm8533.01}
Let \(x\geq 1\) be a large number, and $a\leq q=o(\log x)$ be a pair of integers. Then, $\omega(n)$ has the followings average orders in an arithmetic progression.
\begin{enumerate} [font=\normalfont, label=(\roman*)]
\item  
$\displaystyle 
\sum_{\substack{n \leq x\\ n \equiv a \bmod q}}\omega(n) =\frac{1}{\varphi(q)}x \log \log x +x\beta(q,a)+ O\left (\frac{x}{\log x}\right ), 
$
\item 
$\displaystyle
\sum_{\substack{n \leq x\\ n \equiv a \bmod q}}\left (\omega(n)-\log \log n\right )^2 \leq \frac{C(q,a)}{\varphi(q)}x \log \log x,    
$
\end{enumerate}
where $\beta(q,a)\ne 0$ and $C(q,a)$ are constants.
\end{thm}
\begin{proof} (i) Let $\{x\} \in (0,1)$ be the fractional function. The finite sum $\sum_{k \leq x/p}1$ tallies the number of integers $n \leq x$ divisible by a prime $p \leq x$. Thus, 
\begin{eqnarray} \label{eq8533.55}
\sum_{\substack{n \leq x\\ n \equiv a \bmod q}}\omega(n) &= & \sum_{\substack{p \leq x\\ p \equiv a \bmod q}}\sum_{k \leq x/p}1\\
&= & x\sum_{\substack{p \leq x\\ p \equiv a \bmod q}}\frac{1}{p} -\sum_{\substack{p \leq x\\ p \equiv a \bmod q}}\left \{ \frac{x}{p} \right \}\nonumber.
\end{eqnarray}
Apply Mertens theorem in arithmetic progression to the first finite sum, and estimate the second finite sum to obtain this: 
\begin{equation} \label{eq8533.57}
\sum_{\substack{n \leq x\\ n \equiv a \bmod q}}\omega(n) =  
 x\left (\frac{1}{\varphi(q)} \log \log x +\beta(q,a)+ O\left (\frac{1}{\log x}\right )\right ) + O\left (\frac{x}{\varphi(q)\log x}\right ),
\end{equation} where $\beta(q,a)\ne 0$ is a constant.
\end{proof}
Observe that there are a few versions of Mertens theorem in arithmetic progression, see \cite[Theorem 15.4]{DL12}, \cite{LZ07}, et alii. The basic case $q=2$ of Theorem \ref{thm8533.01} is proved in \cite[Theorem 7.2]{DL12}, \cite[Proposition 2.6]{OM14}, et cetera. The more general concept of the Erdos-Kac theorem provides finer details on the distribution of the random variable $\omega(n) \in \N$.

\begin{lem} \label{lem8533.05}
Let \(n\geq 1\) be a large integer, then
\begin{enumerate} [font=\normalfont, label=(\roman*)]
\item  \text{The average number $\omega(n)$ of prime divisors $p \mid n$ satisfies}
$$
\omega(n) \ll \log \log n . 
$$
\item \text{The maximal number $\omega(n)$ of prime divisors $p \mid n$ satisfies}
$$
\omega(n) \ll \log n/ \log \log n .   
$$
\end{enumerate}
\end{lem}
\begin{proof} (i) Set $a=1$ and $q=2$ in Theorem \ref{thm8533.01}-i. (ii) Set $n=\prod_{p \leq x} p$, and employ routine calculations. 
\end{proof}
Both of these results are standard results in analytic number theory, see \cite[Theorem 2.6]{MV07}.

\subsection{Mobius Function} \label{s297}

\begin{dfn}  \label{dfn297.20}{\normalfont The Mobius function $\mu:\mathbb{N} \longrightarrow \{-1,0,1\}$ is defined by
\begin{equation}
\mu(n) =
\left \{
\begin{array}{ll}
(-1)^{\omega(n)}     &n=p_1 p_2 \cdots p_v\\
0           &n \ne p_1 p_2 \cdots p_v,\\
\end{array}
\right .
\end{equation}
where the $p_i\geq 2$ are primes. }
\end{dfn}

The function $\mu$ is quasiperiodic. It has a period of 4, that is, $\mu(4)= \cdots =\mu(4m)=0$ for any integer $m \in \Z$. But, its interperiods values are pseudorandom, that is, the values 
\begin{equation} \label{eq297.100}
\mu(n),\quad \mu(n+4), \quad \cdots,\quad \mu(n+4m)
\end{equation}
are not periodic as $n \to \infty$. 

\begin{dfn}  \label{dfn297.25}{\normalfont An integer $n \in \N$ is said to be $s$-power free if for each prime $p \mid n$, the maximal prime power divisor is $p^{s-1} \mid \mid n$. Equivalently, the $p$-\textit{adic} valuation $v_p(n)=s-1$ for any $s\geq 2$.
}
\end{dfn}
The $2$-free integers are usually called squarefree integers.
\begin{dfn}  \label{dfn297.30}{\normalfont The characteristic function for $s$-power free integers is defined by
\begin{equation}
\mu_s(n) =
\left \{
\begin{array}{ll}
1     &\text{if } p^s \nmid n \text{ for any prime } p \mid n,\\
0           &\text{if } p^s \mid n \text{ for any prime } p \mid n.\\
\end{array}
\right .
\end{equation}
 }
\end{dfn}

The characteristic function for $s$-power free integers is closely linked to the Mobius function. 
\begin{lem} \label{lem297.58} For any integer $n \geq 1$, the characteristic function for squarefree integers has the expansion
\begin{equation}
\mu(n)^2= \sum_{d^2 \mid n} \mu(d).
\end{equation}
More generally, the characteristic function for $s$-power free integers has the expansion
\begin{equation}
\mu_s(n)= \sum_{d^s \mid n} \mu(d).
\end{equation} 
\end{lem}
The case $s=2$ for squarefree integers is usually denoted by $\mu^2(n)=\mu_2(n)$. Some early works on this topic appear in \cite{CL32} and \cite{ML47}. 

\begin{dfn}  \label{dfn297.35}{\normalfont A pair of integers $a$ and $q$ are relatively prime if and only if $\gcd(a,q)=1$. The characteristic function for relatively prime integers is defined by
\begin{equation}
\sum_{\substack{d \mid a \\ d \mid q}} \mu(d) =
\left \{
\begin{array}{ll}
1     &\text{if and only if } \gcd(a,q)=1,\\
0           &\text{if and only if } \gcd(a,q)\ne 1.\\
\end{array}
\right .
\end{equation}
 }
\end{dfn}

\begin{lem} \label{lem297.98}Let \(n\geq 1\) be an integer. Then, 
\begin{enumerate} [font=\normalfont, label=(\roman*)]
\item  
$\displaystyle 
\sum_{d\mid n}\mu(n) =
\left \{
\begin{array}{ll}
1     &\text{if } n=1,\\
0           &\text{if } n\ne 1.\\
\end{array}
\right .
$
\item 
$\displaystyle
\sum_{d\mid n}\mu^2(n) =2^{\omega(n)}.  
$
\end{enumerate}

\end{lem}

\subsection{Extreme Values Of The Totient Function} \label{s997}
Some estimates for the extreme values of the Euler totient function are stated in this subsection. The Euler totient function counts the number of relatively prime integers \(\varphi (n)=\#\{ k:\gcd (k,n)=1 \}\). For each $n\in \mathbb{N}$, this counting function is compactly expressed by the analytic formula
\begin{equation}\label{eq997.22}
\varphi (n)=n\sum_{d \mid n} \frac{\mu(d)}{d}=n\prod_{p \mid n}\left (1-\frac{1}{p}\right ).
\end{equation}
The explicit lower bound
\begin{equation}\label{eq997.30}
\frac{\varphi(n)}{n}> \frac{1}{e^{\gamma} \log \log n+5/(2 \log \log n)}
\end{equation} 
and other estimates are given in \cite[Theorem 7]{RS62}. The maximal values of the Euler function occurs at the prime arguments. Id est, $\varphi(p)=p-1<p$. There are other subsets of integers that have nearly maximal values. In fact, asymptotically, these integers and the primes number have the same order of magnitudes.\\

\begin{lem} \label{lem997.51} Let $ x \geq 1$ be a large number, and let $n=1+\prod_{p \leq \log x}p$. Then
\begin{enumerate} [font=\normalfont, label=(\roman*)]
\item $\varphi(n)=n+O\left (n/\log \log n \right)$,
\item $\varphi(n+1)=n/2+O\left (n/\log n \right)$.
\end{enumerate}		 
\end{lem}

\begin{proof} (i) Observe that $\log n \geq \sum_{p \leq \log x} \log p$, so that $p \leq \log x \leq 2\log n$. Hence, a prime divisor $q \mid n=1+\prod_{p \leq \log x}p$ implies that $q>\log n$. Consequently, there is the upper bound
\begin{eqnarray}\label{eq997.11}
\varphi(n)&= &n \prod_{p\mid n}\left ( 1-\frac{1}{p}\right)\\
&\leq & n\left ( 1-\frac{1}{\log n}\right)\nonumber\\
&=& n +O\left (\frac{1}{\log n}\right) \nonumber.
\end{eqnarray}
In the other direction, there is the lower bound
\begin{eqnarray}\label{eq997.13}
	\varphi(n)&=&n \prod_{p \mid n}\left ( 1-\frac{1}{p}\right)\\
&\geq &n \prod_{\log n< p \leq 2\log n}\left ( 1-\frac{1}{p}\right)\nonumber\\
&=&n+O\left ( \frac{n}{\log \log n} \right )\nonumber.
	\end{eqnarray}
	Both relations \eqref{eq997.11} and \eqref{eq997.13} confirm the claim.
	(ii) The prime divisors of $n+1$ are $q=2$ and some prime $q> \log n$, so the claim follows from
	\begin{equation} \label{eq997.15}
	\varphi(n+1)=(n+1) \prod_{p|(n+1)}\left ( 1-\frac{1}{p}\right)\leq \frac{n}{2}\left (1-\frac{1}{\log n} \right )=\frac{n}{2}+O\left ( \frac{n}{\log n} \right ).
	\end{equation} 
	   \end{proof}

\begin{thm} \label{thm9192.21}  Let \(p\geq 2\) be a large prime. Then, the followings extreme values hold.
\begin{enumerate} [font=\normalfont, label=(\roman*)]
\item$ \displaystyle 
\frac{\varphi(n)}{n}\leq n-1$,  \tabto{6cm} if $n\geq 2$ is an integer.
\item$ \displaystyle 
\frac{\varphi(n)}{n}\geq \frac{e^{-\gamma}}{4\log \log n}$,  \tabto{6cm} if $n\geq 2$ is a highly composite integer.
\item$ \displaystyle 
\frac{\varphi(n)}{n}\approx \frac{e^{-\gamma}}{\log \log \log n}$,    \tabto{6cm} if $n\geq 2$ is an average integer.
\end{enumerate}
 \end{thm}

The totient function have a wide range of values, as confirmed by Lemma \ref{lem997.51}, and this accounts for the wide range and large gaps in the sequence of totient gaps  
\begin{equation} \label{eq997.17}
	\varphi(2)-\varphi(1), \; \;	\varphi(3)-\varphi(2),	\; \;\varphi(4)-\varphi(3),\; \; \ldots, \; \;	\varphi(n+1)-\varphi(n), \; \; \ldots .
	\end{equation}
The gap can be as small as $\varphi(n+1)-\varphi(n)=0$, and it can be as large as $\varphi(n+1)-\varphi(n)=n/2+O\left (n/\log n \right)$. For example, $\varphi(4)-\varphi(3)=0$, and $\varphi(2\cdot3\cdot5+1)-\varphi(2\cdot3\cdot5+2)=14$.


\section{Summatory Functions For Squarefree Integers} \label{s9339}
The subset of $2$-power free integers are usually called squarefree  integers, and denoted by 
\begin{equation}
\mathcal{Q}_2=\{n\in \mathbb{Z}:\mu^2(n)\ne 0\}
\end{equation} 
and the complementary subset of non squarefree integers is denoted by 
\begin{equation}
\overline{\mathcal{Q}_2}=\{n\in \mathbb{Z}:\mu^2(n)= 0\}.
\end{equation} 
The number of squarefree integers have the following asymptotic formulas.
\begin{lem} \label{lem9339.107} Let $\mu: \mathbb{Z} \longrightarrow \{-1,0,1\}$ be the Mobius function. Then, for any sufficiently large number $x\geq1$, 
\begin{equation} 
\sum_{n \leq x} \mu^2(n) =\frac{6}{\pi^2}x+O \left (x^{1/2} \right ). 
\end{equation} 
\end{lem}
\begin{proof} Use Lemma \ref{lem297.58} or confer to the literature.\end{proof}

The constant coincides with the density of squarefree integers. Its approximate numerical value is
\begin{equation}\label{eq9339.72}
\frac{6}{\pi^2}=\prod_{q\geq 2}\left ( 1-\frac{1}{q^2}\right )=0.607988295164627617135754\ldots,
\end{equation}
where $q\geq2$ ranges over the primes. The remainder term
\begin{equation} 
R(x)=\sum_{n \leq x} \mu^2(n) -\frac{6}{\pi^2}x
\end{equation} 
is a topic of current research, its optimum value is expected to satisfies the upper bound $R(x)=O(x^{1/4+\varepsilon})$ for any small number $\varepsilon>0$. Currently, $R(x)=O\left (x^{1/2}e^{-\sqrt{\log x}}\right )$ is the best unconditional remainder term.

\begin{lem} \label{lem9339.8} Let $\mu(n)$ be the Mobius function. Then, for any sufficiently large number $x\geq1$, 
\begin{equation} 
\sum_{n \leq x} \mu^2(n) =\frac{6}{\pi^2}x+\Omega \left (x^{1/4} \right ). 
\end{equation} \end{lem}
\begin{proof} The generating series for squarefree integers is $\zeta(s)/\zeta(2s)=\sum_{n \geq 1}\mu^2(n)n^{-s}$ at $s=2$. The Perron intergral yields
\begin{equation}\label{eq9339.66}
\sum_{n \leq x} \mu^2(n) =\frac{1}{i2 \pi}\int_{c-\infty}^{c+\infty}\frac{\zeta(s)}{\zeta(s)}\frac{x^s}{s}ds=\frac{1}{\zeta(2)}x +\sum_{\zeta(\rho)=0}c_{\rho}x^{\rho/2},
\end{equation}
where $c\ne0$ is a constant. The coefficients $c_{\rho}$ are indexed by the zeros $\rho \in \C$ of the zeta function $\zeta(s)$. Since the zeta function has a zero $\rho_0=1/2+i14.134725\ldots $, the claim follows.
\end{proof}
\begin{thm}\label{thm9339.10}  Let $x\geq 1$ be a large number, let $a$ and $q$ be a pair of integers, $1 \leq a <q =O(\log^c x )$, with $c\geq 0$ constant, and let $\mu: \mathbb{Z} \longrightarrow \{-1,0,1\}$ be the Mobius function. Then, 
\begin{equation}\label{eq9339.74}
\sum_{\substack{n \leq x \\
n \equiv a \bmod q}}\mu(n)^2=\frac{6}{\pi^2}\prod_{p\mid q}\left ( 1-\frac{1}{p^2}\right )^{-1}\frac{x}{q}+O\left (\frac{x}{q} +q^{1/2+\varepsilon} \right ),
\end{equation}
where $\varepsilon>0$ is an arbitrary small number.
\end{thm}

\begin{proof} Consult \cite{HC75}, \cite{WR80}, and the literature.\end{proof}
The range of moduli $q \leq x^{2/3}$ is discussed and improved to $q \leq x^{1-\varepsilon}$ in \cite{NR14}. The $q$-dependence in the constant 
\begin{equation}\label{eq9339.76}
\frac{1}{q}\sum_{\substack{n \geq 1 \\
\gcd(n,q)=1}} \frac{\mu(n)}{n^2}=\frac{1}{q}\prod_{p \nmid q}\left ( 1-\frac{1}{p^2}\right )=\frac{6}{\pi^2}\frac{1}{q}\prod_{p\mid q}\left ( 1-\frac{1}{p^2}\right )^{-1}
\end{equation}
propagates the dependence in the asymptotic formula for consecutive $s$-power free integers. For example, the  probability or density of two consecutive squarefree integers is not $\left (6/\pi^2 \right )^2$, but a more complicated expression similar to \eqref{eq9339.76}. The equidistribution of $s$-power free integers in arithmetic progressions is affirmed by the result below. This also indicates a level of distribution of $2/3$ over any arithmetic progression $\{n=qm+a: m \geq 1\}$. 

\begin{thm}\label{thm3.110}  Let $x\geq 1$ be a large number, let $a$ and $q$ be a pair of integers, $1 \leq a <q =O(\log^c x )$, with $c\geq 0$ constant, and let $\mu: \mathbb{Z} \longrightarrow \{-1,0,1\}$ be the Mobius function. Then, 
\begin{equation}\label{eq9339.174}
\sum_{q \leq x^{2/3}\log^{-c-1}x} \max_{ a \bmod q} \left | \sum_{\substack{n \leq x \\
n \equiv a \bmod q}}\mu(n)^2-\frac{\varphi(q)}{d\varphi(q/d)}\prod_{p\nmid q}\left ( 1-\frac{1}{p^2}\right )\frac{x}{q} \right | \ll \frac{x}{\log ^c x} ,
\end{equation}
where $d=\gcd(a,q)$ and $c>0$ is an arbitrary constant.
\end{thm}

\begin{proof} Consult \cite{OR71} and the literature.\end{proof}

\begin{lem} \label{lem9339.117} Let $x\geq1$ be a large number, and let $\mu: \mathbb{Z} \longrightarrow \{-1,0,1\}$ be the Mobius function. If $q=O(\log^c x)$ with $c\leq 0$ constant, then,  
\begin{equation} 
\sum_{\substack{n \leq x\\ \gcd(n,q)=1}} \mu^2(n) =\frac{6}{\pi^2}\prod_{p\nmid q}\left ( 1+\frac{1}{p}\right )^{-1}x+O \left (x^{1/2} \right ). 
\end{equation} 
\end{lem}
\begin{proof} The proof is lengthier and more difficult than Lemma \ref{lem9339.107}, see \cite[Lemma 2]{DK05}.\end{proof}

\section{Correlation Functions For Squarefree Integers} \label{s8009}
A sequence of squarefree integers
\begin{equation} \label{eq8009.030}
n +a_0, \quad n+a_1,\quad n+a_2,\quad \ldots,\quad n+a_k, 
\end{equation}
imposes certain restriction on the $(k+1)$-tuple $( a_0, a_1, \ldots, a_k)$. A stronger restriction is required for sequence of prime $(k+1)$-tuples , see \cite{BT13}, and the literature for extensive details.

\begin{dfn}  \label{dfn8009.35}{\normalfont A $k$-tuple $(a_0,a_1, \ldots, a_k)$ is called \textit{admissible} if the numbers $a_0,a_1, \ldots, a_k$ is not a complete residues system modulo $p$ for any prime $p\leq k$.
 }
\end{dfn}

\begin{lem}\label{lem8009.41}  Let $x\geq 1$ be a large number, and let $\mu: \mathbb{Z} \longrightarrow \{-1,0,1\}$ be the Mobius function. Then, 
\begin{equation}\label{eq8009.74}
\sum_{n \leq x}\mu(n)^2 \mu(n+1)^2=\prod_{p\geq 2}\left ( 1-\frac{2}{p^2}\right )x+O\left (x^{2/3} \right ).
\end{equation}
\end{lem}

\begin{proof} The earliest proof seems to be that in \cite{CL32}, and \cite{ML47}. Recent proofs appear in \cite{MI17}, and the literature.
\end{proof}

The constant coincides with the density of 2-consecutive squarefree integers. Its approximate numerical value is
\begin{equation}\
\prod_{q\geq 2}\left ( 1-\frac{2}{q^2}\right )=0.322699054242535576161483\ldots,
\end{equation}
where $q\geq2$ ranges over the primes.
\begin{lem}\label{lem8009.44}  Let $x\geq 1$ be a large number, and let $\mu: \mathbb{Z} \longrightarrow \{-1,0,1\}$ be the Mobius function. Then, 
\begin{equation}\label{eq8009.74}
\sum_{n \leq x}\mu(n)^2 \mu(n+1)^2\mu(n+2)^2=\prod_{p\geq 2}\left ( 1-\frac{3}{p^2}\right )x+O\left (x^{2/3} \right ).
\end{equation}
\end{lem}
The earliest result in this direction appears to be
\begin{equation}\label{eq8009.74}
\sum_{n \leq x}\mu(n)^2 \mu(n+t)^2=cx+O\left (x^{2/3} \right ),
\end{equation}
where $c>0$ is the constant \eqref{eq9339.72}, is studied in \cite{ML47}. Except for minor adjustments, the generalization to sequences of $(k+1)$-tuples of squarefree integers has the same structure.

\begin{thm}\label{thm8009.10}  Let $ a\geq 1$ and $s\geq 2$ be small integers. Let $x\geq 1$ be a large number, and let $\mu_s: \mathbb{Z} \longrightarrow \{-1,0,1\}$ be the $s$-power free characteristic function. Then, 
\begin{equation}\label{eq8009.74}
\sum_{n \leq x}\mu(n+a_0)^2\mu(n+a_1)^2 \cdots \mu(n+a_k)^2=\prod_{p\geq 2}\left ( 1-\frac{\rho(s)}{p^2}\right )x+O\left (x^{2/3+\varepsilon} \right ),
\end{equation}
where $q\geq 1$ is a constant, and 
\begin{equation}\label{eq8009.79}
\rho(s)=\#\{m\leq p^2: qm+a_i\equiv 0 \bmod p^2 \text{ for } i=0,1,2, ..., k\},
\end{equation}
and $\varepsilon>0$ is an arbitrary small number depending on $k$ and $q$.
\end{thm}
\begin{proof}Consult \cite{ML47}, \cite[Theorem 1.2]{MI17}, \cite{TK85}, and the literature.\end{proof}

The literature does not seem to offer any results for squarefree twin integers $n$ and $n+a$, which are relatively prime to $q=q(a)$. A plausible result might have the form given below. 
\begin{conj}\label{conj8009.105}  Let $x\geq 1$ be a large number, and let $\mu: \mathbb{Z} \longrightarrow \{-1,0,1\}$ be the Mobius function. If $a\geq 1$ is a fixed integer, and $q=O(\log^c x)$ with $c\geq 0$ constant, then, 
\begin{equation}\label{eq8009.190}
\sum_{\substack{n \leq x\\ \gcd(n,q)=1\\\gcd(n+a,q)=1}}\mu(n)^2 \mu(n+a)^2=c_2(q,a)\prod_{p\nmid q}\left ( 1+\frac{1}{p}\right )^{-2}\prod_{p\geq 2}\left ( 1-\frac{2}{p^2}\right )x+O\left (x^{1-\delta} \right ),
\end{equation}
where dependence correction factor $c_2(q,a)\geq 0$, and $\delta >0$ is a small number.
\end{conj}

The dependence correction factor $c_2(q,a)\geq 0$, and the parameter $q=q(a)$ depends on $a \geq 1$. For instance, for $a=2b+1$ odd, the value $q=q(a)$ must be odd, and $c_2(q,a)>0$, otherwise $c_2(q,a)= 0$ for even $q$.

\section{Summatory Functions For $s$-Power Free Integers} \label{s9539}
The subset of $k$-power free integers is usually denoted by 
\begin{equation}
\mathcal{Q}_s=\{n\in \mathbb{Z}:\mu_s(n)\ne 0\}
\end{equation} 
and the complementary subset of non $s$-free integers is denoted by 
\begin{equation}
\overline{\mathcal{Q}_s}=\{n\in \mathbb{Z}:\mu_s(n)= 0\}.
\end{equation}

The number $s$-power free integers have the following asymptotic.
\begin{lem} \label{lem9539.118} Given an integer $s \geq 2$, let $\mu_s(n)$ be the $s$th-Mobius function. Then, for any sufficiently large number $x\geq1$, 
\begin{equation} 
\sum_{n \leq x} \mu_s(n) =\frac{1}{\zeta(s)}x+O \left (x^{1/s} \right ).  
\end{equation} \end{lem}

\begin{proof}  The basic $s$th-Mobius function $\mu_s$ is explained in Definition \ref{dfn297.30}. This result is attributed to Gegenbauer, 1885. Recent proofs are provided in \cite{IA03} and the literature.\end{proof}

\begin{lem} \label{lem9339.108}  Given an integer $s \geq 2$, let $\mu_s(n)$ be the $s$th-Mobius function. Then, for any sufficiently large number $x\geq1$, 
\begin{equation} 
\sum_{n \leq x} \mu_s(n) =\frac{1}{\zeta(2s)}x+\Omega \left (x^{1/2s} \right ). 
\end{equation} \end{lem}
\begin{proof} Same as the proof of Lemma \ref{lem9339.8}, mutatis mutandus.
\end{proof}

\begin{conj} \label{conj9339.208}  Given a pair of integers $s \geq 2$, and $q\geq 2$, let $\mu_s(n)$ be the $s$th-Mobius function. Then, for any sufficiently large number $x\geq1$, 
\begin{equation} 
\sum_{\substack{n \leq x\\ \gcd(n,q)=1}} \mu_s(n) =\frac{1}{\zeta(2s)}\prod_{p\nmid q}\left ( 1+\frac{1}{p}\right )^{-1}x+O \left (x^{1/2s} \right ). 
\end{equation} \end{conj}

\section{Correlation Functions For $s$-Power Free Integers} \label{s9549}
\begin{thm}\label{thm9549.310}  Let $s \geq 2)$ be an integer. Let $x\geq 1$ be a large number, and let $\mu_s: \mathbb{Z} \longrightarrow \{-1,0,1\}$ be the characteristic function of $s$-power free integers. Then, 
\begin{equation}\label{eq9539.374}
\sum_{n \leq x}\mu_s(n)\mu_s(n+a) =\prod_{p\geq 2}\left ( 1-\frac{\rho(p,a)}{p^s}\right )x+O\left (x^{\alpha(s)+\varepsilon} \right ),
\end{equation}
where 
\begin{equation} \label{eq9549.376}
\rho(p) =
\left \{
\begin{array}{ll}
2     &\text{ if } p^s \nmid a,\\
1           &\text{ if } p^s \mid a,\\
\end{array}
\right .
\end{equation}
and 
\begin{equation} \label{eq9549.376}
\alpha(p,a) =\frac{14}{7s+8}
\end{equation}
and $\varepsilon>0$ is an arbitrary small number.
\end{thm}
\begin{proof} Different proofs are given in \cite{RT12}, \cite[Theorem 1.2]{BJ13}, which have slightly different remainder terms.\end{proof}

The main problems in this area are the determination of the best remainder terms for various summatory functions. For instance, the remainder term
\begin{equation} 
R_s(x)=\sum_{n \leq x} \mu_s(n) -\frac{1}{\zeta(2s)}x
\end{equation} 
in Theorem \ref{thm9549.310} is expected to satisfies the upper bound $R_s(x)=O(x^{1/2s+\varepsilon})$ for any small number $\varepsilon>0$. A survey of the literature on $s$-power free integers and arithmetic functions is presented in \cite{PF05}. Currently, $R_s(x)=O\left (x^{1/2s}e^{-\sqrt{\log x}}\right )$ is the best unconditional remainder term.\\

The literature does not seem to offer any results for $s$-power free twin integers $n$ and $n+a$, with $a\geq 1$. A plausible result might have the form given below.

\begin{conj}\label{conj9549.100}  Given a pair of integers $a \geq 1$ and $s\geq 2$. Let $x\geq 1$ be a large number, and let $\mu: \mathbb{Z} \longrightarrow \{-1,0,1\}$ be the Mobius function. If $a \geq 1$, and $q=O(\log^c x)$ with $c\geq 0$ constant, then, 
\begin{equation}\label{eq8009.190}
\sum_{\substack{n \leq x\\ \gcd(n,q)=1\\\gcd(n+a,q)=1}}\mu_s(n) \mu_s(n+1)=c_s(q,a)\prod_{p\nmid q}\left ( 1+\frac{1}{p}\right )^{-s}\prod_{p\geq 2}\left ( 1-\frac{2}{p^s}\right )x+O\left (x^{1/2s-\delta} \right ),
\end{equation}
where $c_s(q,a)\geq 0$ is a constant, and $\delta >0$ is a small number.
\end{conj}

The constant $c_s(q,a)\geq 0$ and the parameter $q=q(a)$ depend on $a \geq 1$. For instance, for $a=2b+1$ odd, the value $q=q(a)$ must be odd, and $c_s(q,a)>0$, otherwise $c_s(q,a)= 0$ for even $q$.

\section{Probabilities For Consecutive Squarefree Integers}\label{s1180}
The events of 2 consecutive squarefree integers $X_0$ and $X_1$ are dependent random variables. Similar, the events of 3 consecutive squarefree integers $X_0$, $X_1$, and $X_2$ are dependent random variables.  

The probability $P(\mu(X_0)=\pm1, \mu(X_1)=\pm1)$ for 2 consecutive squarefree integers is asymptotic to the constant attached to the main term in Lemma \ref{lem8009.41}. Specifically,
\begin{equation} \label{eq1180.040}
\prod_{q \geq 2}\left (1-\frac{2}{q^2}\right )=\left (\frac{6}{\pi^2}\right )^2 \prod_{q \geq 2}\left (1+\frac{1}{q^2(q^2-2)}\right )^{-1}=0.322699054242535576161483\ldots.
\end{equation}
The reduction from independent events is measured by the dependence correction factor
\begin{equation} \label{eq1180.048}
c_2(2)=\prod_{q \geq 2}\left (1+\frac{1}{q^2(q^2-2)}\right )^{-1}=0.872985953449313618771745\ldots.
\end{equation}

The probability $P(\mu(X_0)=\pm1, \mu(X_1)=\pm1, \mu(X_2)=\pm1)$ for 3 consecutive squarefree integers is asymptotic to the constant attached to the main term in Lemma \ref{lem8009.44}. Specifically,
\begin{equation} \label{eq1180.140}
\prod_{q \geq 2}\left (1-\frac{3}{q^2}\right )=\left (\frac{6}{\pi^2}\right )^3\prod_{q \geq 2}\left (1+\frac{3q^2-1}{q^4(q^2-3)}\right )^{-1}=0.125524878896821220184683\ldots.
\end{equation}
The reduction from independent events is measured by the dependence correction factor
\begin{equation} \label{eq1180.148}
c_2(3)=\prod_{q \geq 2}\left (1+\frac{3q^2-1}{q^4(q^2-3)}\right )^{-1}=0.558526979127689105533330\ldots.
\end{equation}
Accordingly, consecutive squarefree integers are highly correlated.


\section{Primitive Roots Test} \label{969}
For a prime $p \geq 2$, the multiplicative group of the finite fields $\mathbb{F}_p$ is a cyclic group for all primes. 

\begin{dfn} { \normalfont The order $\min \{k \in \mathbb{N}: u^k \equiv 1 \bmod p \}$ of an element $u \in \mathbb{F}_p$ is denoted by $\ord_p(u)$. An element is a \textit{primitive root} if and only if $\ord_p(u)=p-1$. }
\end{dfn}
The Euler totient function counts the number of relatively prime integers \(\varphi (n)=\#\{ k\leq n:\gcd (k,n)=1 \}\). 

\begin{lem} {\normalfont (Fermat-Euler)} \label{lem2.1}If \(a\in \mathbb{Z}\) is an integer such that \(\gcd (a,n)=1,\) then \(a^{\varphi (n)}\equiv
	1 \bmod n\).
\end{lem}
\begin{lem} \label{lem969.05}  {\normalfont (Primitive root test)} An integer $u \in \Z$ is a primitive root modulo an integer $n \in \N$ if and only if 
\begin{equation}\label{eq969.52}
u^{\varphi (n)/p} -1\not \equiv 0 \mod  n
\end{equation}
for all prime divisors $p \mid \varphi (n)$.
\end{lem}
The primitive root test is a special case of the Lucas primality test, introduced in \cite[p.\ 302]{ LE78}. A more recent version appears in \cite[Theorem 4.1.1]{CP05}, and similar sources. 

\begin{lem} \label{lem969.21}  {\normalfont (Complexity of primitive root test)} Given a prime $p \geq 2$, and primes decomposition of the squarefree part $p_1 p_2 \cdots p_v \mid p-1$, a primitive root modulo $p$ can be determined in deterministic polynomial time $O(\log ^c p)$, some constant $c >1$.
\end{lem}
\begin{proof} The mechanics of the deterministic polynomial time algorithm are specified in \cite[Chapter 11]{SV08}. By \cite[Theorem 1.2]{CN18}, the algorithm is repeated at most $O\left (  (\log p)^{1+\varepsilon} \right )$ times for each $u=O\left (  (\log p)^{1+\varepsilon} \right )$, with $\varepsilon>0$. These prove the claim.
\end{proof}

\section{Representations of the Characteristic Functions} \label{s333}
The characteristic function \(\Psi :G\longrightarrow \{ 0, 1 \}\) of primitive elements is one of the standard analytic tools employed to investigate the various properties of primitive roots in cyclic groups \(G\). Many equivalent representations of the characteristic function $\Psi $ of primitive elements are possible. Several of these representations are studied in this section.

\subsection{Divisors Dependent Characteristic Function}
A representation of the characteristic function dependent on the orders of the cyclic groups is given below. This representation is sensitive to the primes decompositions $q=p_1^{e_1}p_2^{e_2}\cdots p_t^{e_t}$, with $p_i$ prime and $e_i\geq1$, of the orders of the cyclic groups $q=\# G$. \\

\begin{lem} \label{lem333.02}
Let \(G\) be a finite cyclic group of order \(p-1=\# G\), and let \(0\neq u\in G\) be an invertible element of the group. Then
\begin{equation}
\Psi (u)=\frac{\varphi (p-1)}{p-1}\sum _{d \mid p-1} \frac{\mu (d)}{\varphi (d)}\sum _{\ord(\chi ) = d} \chi (u)=
\left \{\begin{array}{ll}
1 & \text{ if } \ord_p (u)=p-1,  \\
0 & \text{ if } \ord_p (u)\neq p-1. \\
\end{array} \right .
\end{equation}
\end{lem}

\begin{proof} A full proof appears in \cite[Lemma 3.1]{CN18}.
\end{proof}
There are other proofs in the literature. Some details on this characteristic function are given in \cite[p. 863]{ES57}, \cite[p.\ 258]{LN97}, \cite[p.\ 18]{MP07}. The works in \cite{DH37}, and \cite{WR01} attribute this formula to Vinogradov. But some authors make an earlier reference to Landau.\\

The characteristic function for multiple primitive roots is used in \cite[p.\ 146]{CZ98} to study consecutive primitive roots. In \cite{DS12} it is used to study the gap between primitive roots with respect to the Hamming metric. And in \cite{WR01} it is used to prove the existence of primitive roots in certain small subsets \(A\subset \mathbb{F}_p\). In \cite{DH37} it is used to prove that some finite fields do not have primitive roots of the form $a\tau+b$, with $\tau$ primitive and $a,b \in \mathbb{F}_p$ constants. In addition, the Artin primitive root conjecture for polynomials over finite fields was proved in \cite{PS95} using this formula.

\subsection{Divisors Free Characteristic Function}
It often difficult to derive any meaningful result using the usual divisors dependent characteristic function of primitive elements given in Lemma \ref{lem333.02}. This difficulty is due to the large number of terms that can be generated by the divisors, for example, \(d\mid p-1\), involved in the calculations, see \cite{ES57}, \cite{DS12} for typical applications and \cite[p.\ 19]{MP04} for a discussion. \\
	
A new \textit{divisors-free} representation of the characteristic function of primitive element is developed here. This representation can overcomes some of the limitations of its counterpart in certain applications. The \textit{divisors representation} of the characteristic function of primitive roots, Lemma \ref{lem333.02}, detects the order \(\ord_p (u)\) of the element \(u\in \mathbb{F}_p\) by means of the divisors of the totient \(p-1\). In contrast, the \textit{divisors-free representation} of the characteristic function, Lemma \ref{lem333.03}, detects the order \(\text{ord}_p(u) \geq 1\) of the element \(u\in \mathbb{F}_p\) by means of the solutions of the equation \(\tau ^n-u=0\) in \(\mathbb{F}_p\), where \(u,\tau\) are constants, and \(1\leq n<p-1, \gcd (n,p-1)=1,\) is a variable. 

\begin{lem} \label{lem333.03}
Let \(p\geq 2\) be a prime, and let \(\tau\) be a primitive root mod \(p\). If \(u\in\mathbb{F}_p\) is a nonzero element, and \(\psi \neq 1\) is a nonprincipal additive character of order \(\ord \psi =p\), then
\begin{equation}
\Psi (u)=\sum _{\gcd (n,p-1)=1} \frac{1}{p}\sum _{0\leq m\leq p-1} \psi \left ((\tau ^n-u)m\right)=\left \{
\begin{array}{ll}
1 & \text{ if } \ord_p(u)=p-1,  \\
0 & \text{ if } \ord_p(u)\neq p-1. \\
\end{array} \right .
\end{equation}
\end{lem}

\begin{proof} A full proof appears in \cite[Lemma 3.2]{CN18}.
\end{proof}

\subsection{Arbitrary Subset Characteristic Function}
The previous construction easily generalize to arbitrary subset of the ring $\Z/p\Z$, and other rings.
\begin{lem} \label{lem333.09}
Let \(p\geq 2\) be a prime, and let $\mathcal{A} \subset \Z/p\Z$ be an arbitrary subset. Let \(\psi \neq 1\) be a nonprincipal additive character of order \(\ord \psi =p\). Then, 
\begin{equation}
\Psi_{\mathcal{A}} (u)=\sum _{x \in \mathcal{A}} \frac{1}{p}\sum _{0\leq m\leq p-1} \psi \left ((x-u)m\right)=\left \{
\begin{array}{ll}
1 & \text{ if } u \in \mathcal{A},  \\
0 & \text{ if } u \not \in \mathcal{A}. \\
\end{array} \right .
\end{equation}
\end{lem}
\begin{proof} Consider the equation
\begin{equation}\label{eq33.34}
x- u=0
\end{equation} 
where $u$ is fixed, and a variable $x \in \mathcal{A}$. Clearly, it has a solution if and only if the fixed element \(u \in \mathcal{A}\). 
\end{proof}

\section{Estimates Of Exponential Sums} \label{s4}
This section provides simple estimates for the exponential sums of interest in this analysis. There are two objectives: To  determine an upper bound, proved in Theorem \ref{thm333.04}, and to show that
\begin{equation} \label{eq3.201}
\sum_{\gcd(n,p-1)=1} e^{i2\pi b \tau^n/p} =\sum_{\gcd(n,p-1)=1} e^{i2\pi \tau^n/p}+E(p),
\end{equation}
where $E(p)$ is an error term, this is proved in Lemma \ref{lem333.22}. 
\subsection{Incomplete And Complete Exponential Sums}
\begin{thm}  \label{thm333.02} {\normalfont (\cite{SR73},  \cite{ML72}) }  Let \(p\geq 2\) be a large prime, and let \(\tau \in \mathbb{F}_p\) be an 

element of large multiplicative order $\ord_p(\tau) \mid p-1$. Then, for any $b \in [1, p-1]$,  and $x\leq p-1$,
\begin{equation}
 \sum_{ n \leq x}  e^{i2\pi b \tau^{n}/p} \ll p^{1/2}  \log p.
\end{equation}

\end{thm}

\begin{proof}  A complete proof appears in \cite[Theorem 5.1]{CN18}.
\end{proof}
This seems to be the best possible upper bound. A similar upper bound for composite moduli $p=m$ is also proved, [op. cit., equation (2.29)]. 

\begin{thm}  \label{thm333.04}  Let \(p\geq 2\) be a large prime, and let $\tau $ be a primitive root modulo $p$. Then,
\begin{equation}
	 \sum_{\gcd(n,p-1)=1} e^{i2\pi b \tau^n/p} \ll  p^{1-\varepsilon} 
\end{equation} 
for any $b \in [1, p-1]$, and any arbitrary small number $\varepsilon \in 
(0, 1/2)$. 
\end{thm} 
\begin{proof} A complete proof appears in \cite[Theorem 5.2]{CN18}.
\end{proof}
The upper bound given in Theorem \ref{thm333.04} seems to be optimum. A different proof, which has a weaker upper bound, appears in \cite[Theorem 6]{FS00}, and related results are given in \cite{CC09}, \cite{FS01}, \cite{GZ05}, and \cite[Theorem 1]{GK05}.
\subsection{Equivalent Exponential Sums} 
For any fixed $ 0 \ne b \in \mathbb{F}_p$, the map $ \tau^n \longrightarrow b \tau^n$ is one-to-one in $\mathbb{F}_p$. Consequently, the subsets 
\begin{equation} \label{eq3.220}
 \{ \tau^n: \gcd(n,p-1)=1 \}\quad \text { and } \quad  \{ b\tau^n: \gcd(n,p-1)=1 \} \subset \mathbb{F}_p
\end{equation} have the same cardinalities. As a direct consequence the exponential sums 
\begin{equation} \label{3.330}
\sum_{\gcd(n,p-1)=1} e^{i2\pi b \tau^n/p} \quad \text{ and } \quad \sum_{\gcd(n,p-1)=1} e^{i2\pi \tau^n/p},
\end{equation}
have the same upper bound up to an error term. An asymptotic relation for the exponential sums (\ref{3.330}) is provided in Lemma \ref{lem333.22}. This result expresses the first exponential sum in (\ref{3.330}) as a sum of simpler exponential sum and an error term. 

\begin{lem}   \label{lem333.22}  Let \(p\geq 2\) be a large primes. If $\tau $ be a primitive root modulo $p$, then,
\begin{equation} \sum_{\gcd(n,p-1)=1} e^{i2\pi b \tau^n/p} = \sum_{\gcd(n,p-1)=1} e^{i2\pi  \tau^n/p} + O(p^{1/2} \log^3 p),
\end{equation} 
for any $ b \in [1, p-1]$. 	
\end{lem} 
\begin{proof} A complete proof appears in \cite[Lemma 5.1]{CN18}.
\end{proof}
The same proof works for many other subsets of elements $\mathcal{A} \subset \mathbb{F}_p$. For example, 
\begin{equation}
\sum_{n \in \mathcal{A}} e^{i2\pi b \tau^n/p} = \sum_{n \in \mathcal{A}} e^{i2\pi  \tau^n/p} + O(p^{1/2} \log^c p), 
\end{equation} 
for some constant $c>0$.

\section{Asymptotic Formulas For The Main Terms} \label{s9799}
The notation $f(x)\asymp g(x)$ is defined by $a|f(x)|<|g(x)|<b|f(x)|$ for some constants $a,b >0$. The symbol $f \ll g$ denote $f=O(g)$.

\subsection{Main Term For $k+1$ Consecutive Primitive Roots}
\begin{lem} \label{lem9799.76}  Let \(p\geq 2\) be a large prime, let $k \ll \log p$, and let $\varphi$ be the totient function. Then, 
\begin{equation} \label{eq9799.08}
 \sum_{ n\in \F_p}  \prod_{0 \leq i \leq k}  \left (\frac{1}{p}\sum_{\gcd(n_i,p-1)=1}1  \right )= \left (\frac{\varphi(p-1)}{p-1} \right )^{k+1}p  +O\left (\log^2 p \right ).
	\end{equation} 
\end{lem}

\begin{proof} Each inner sum has the exact value $\varphi(p-1)/p$. Hence,
\begin{eqnarray}\label{eq9799.10}
M(k,p)&=& \sum_{ n\in \F_p} \prod_{0 \leq i \leq k}  \left (\frac{1}{p}\sum_{\gcd(n_i,p-1)=1}1  \right )\nonumber\\
& =&\left (\frac{\varphi(p-1)}{p} \right )^{k+1} \sum_{ n\in \F_p} 1 \\
&=& \left (\frac{\varphi(p-1)}{p} \right )^{k+1}p\nonumber.
\end{eqnarray}
Last, but not least use the readjustment
\begin{equation} \label{eq9729.228}
\frac{\varphi(p-1)}{p}
=\frac{\varphi(p-1)}{p-1}\left ( 1-\frac{1}{p}\right )
\end{equation} 
to obtain the standard form of the main term.
\end{proof}

\subsection{Main Term For $k+1$ Consecutive Squarefree Primitive Roots}
The list of numbers $a_0,a_1, \ldots, a_k$ forms an increasing sequence of distinct integers, an admissible $(k+1)$ tuple, see Definition \ref{dfn8009.35}.
\begin{lem} \label{lem9709.76}  Let \(p\geq 2\) be a large prime, let $k \ll \log p$, and let $\varphi$ be the totient function. Then, 
\begin{equation} \label{eq9709.08}
\sum_{ n\in \F_p} \prod_{0 \leq i \leq k}  \left (\frac{1}{p}\sum_{\gcd(n_i,p-1)=1}\mu(n+a_i)^2 \right ) 
=\prod_{q\geq 2}\left ( 1-\frac{\rho(q)}{q^2}\right )\left (\frac{\varphi(p-1)}{p-1} \right )^{k+1}p+O\left (x^{2/3+\varepsilon} \right ),
	\end{equation} 
where \(\varepsilon >0\) is an arbitrarily small number.
\end{lem}

\begin{proof} Rearrange the finite sum and observe that each inner sum has the exact value $\varphi(p-1)/p=\sum_{\gcd(n,p-1)=1}1$. Hence,
\begin{eqnarray}\label{eq9709.10}
M(k,p)&=& \sum_{ n\in \F_p}  \left (\frac{\mu(n+a_0)^2}{p}\sum_{\gcd(n_0,p-1)=1}1 \cdots  \frac{\mu(n+a_k)^2}{p}\sum_{\gcd(n_k,p-1)=1}1  \right )\nonumber\\
& =&\left (\frac{\varphi(p-1)}{p} \right )^{k+1} \sum_{ n\in \F_p} \mu(n+a_0)^2  \cdots  \mu(n+a_k)^2\\
&=&\prod_{q\geq 2}\left ( 1-\frac{\omega(q)}{q^2}\right )\left (\frac{\varphi(p-1)}{p-1} \right )^{k+1}p+O\left (x^{2/3+\varepsilon} \right )\nonumber.
\end{eqnarray}
The last line follows from Theorem \ref{thm8009.10} applied to the correlation function, and \(\varepsilon >0\) is an arbitrarily small number. Now, use the readjustment
\begin{equation} \label{eq9729.110}
\frac{\varphi(p-1)}{p}
=\frac{\varphi(p-1)}{p-1}\left ( 1-\frac{1}{p}\right )
\end{equation} 
to obtain the standard form of the main term.
\end{proof}

\subsection{Main Term For Squarefree Twin Primitive Roots}
\begin{lem} \label{lem9709.77}  Let \(p\geq 2\) be a large prime, let $\varphi$ be the totient function, and let $\mu$ be the Mobius function. Then, 
\begin{eqnarray} \label{eq9709.08}
&&\sum_{ n\in \F_p}   \left (\frac{\mu^{2}(n)}{p}\sum_{\gcd(n_0,p-1)=1} 1 \right )  \left (\frac{\mu^{2}(n+1)}{p}\sum_{\gcd(n_1,p-1)=1} 1 \right ) \\
&=&\prod_{q \geq 2}\left (1-\frac{2}{q^2}\right )\left (\frac{\varphi(p-1)}{p-1} \right )^2p+ O\left ( p^{2/3}\right ) \nonumber.
\end{eqnarray} 
\end{lem}

\begin{proof}Rearrange it and simplify it as
\begin{eqnarray} \label{eq9040.020}
M_2(2,p)
&=&\sum_{ n\in \F_p}   \left (\frac{\mu^{2}(n)}{p}\sum_{\gcd(n_0,p-1)=1} 1 \right )  \left (\frac{\mu^{2}(n+1)}{p}\sum_{\gcd(n_1,p-1)=1} 1 \right ) \nonumber\\
&=&\left (\frac{\varphi(p-1)}{p} \right )^2\sum_{ n\in \F_p}  \mu^{2}(n)\mu^{2}(n+1) \\
&=&\left (\frac{\varphi(p-1)}{p} \right )^2\left ( \prod_{q \geq 2}\left (1-\frac{2}{q^2}\right )p+ O\left ( p^{2/3}\right )\right )\nonumber
\end{eqnarray}
The last line follows from Lemma \ref{lem8009.41} or Theorem \ref{thm8009.10} applied to the correlation function. Lastly, use the readjustment
\begin{equation} \label{eq9040.022}
\frac{\varphi(p-1)}{p}
=\frac{\varphi(p-1)}{p-1}\left ( 1-\frac{1}{p}\right )
\end{equation} 
to obtain the standard form of the main term.
\end{proof}

\subsection{Main Term For Squarefree Triple Primitive Roots}
\begin{lem} \label{lem9729.78}  Let \(p\geq 2\) be a large prime, let $\varphi$ be the totient function, and let $\mu$ be the Mobius function. Then, the number of three consecutive squarefree primitive root has the asymptotic formula 
\begin{eqnarray} \label{eq9729.08}
&&\sum_{ n\in \F_p}   \left (\frac{\mu^{2}(n)}{p}\sum_{\gcd(n_0,p-1)=1} 1 \right )  \left (\frac{\mu^{2}(n+1)}{p}\sum_{\gcd(n_1,p-1)=1} 1 \right ) 
\left (\frac{\mu^{2}(n+2)}{p}\sum_{\gcd(n_1,p-1)=1} 1 \right ) \nonumber\\
&=&\prod_{q \geq 2}\left (1-\frac{3}{q^2}\right )\left (\frac{\varphi(p-1)}{p-1} \right )^3p+ O\left ( p^{2/3}\right ) .
\end{eqnarray} 
\end{lem}

\begin{proof}Rearrange it and simplify it as
\begin{eqnarray} \label{eq9040.220}
M_2(3,p)
&=&\sum_{ n\in \F_p}   \left (\frac{\mu^{2}(n)}{p}\sum_{\gcd(n_0,p-1)=1} 1 \right )  \left (\frac{\mu^{2}(n+1)}{p}\sum_{\gcd(n_1,p-1)=1} 1 \right )   \nonumber\\
&& \times \left (\frac{\mu^{2}(n+2)}{p}\sum_{\gcd(n_2,p-1)=1} 1 \right )\nonumber\\
&=&\left (\frac{\varphi(p-1)}{p} \right )^3\sum_{ n\in \F_p}  \mu^{2}(n)\mu^{2}(n+1) \mu^{2}(n+2) \\
&=&\left (\frac{\varphi(p-1)}{p} \right )^3\left ( \prod_{q \geq 2}\left (1-\frac{3}{q^2}\right )p+ O\left ( p^{2/3}\right )\right )\nonumber
\end{eqnarray}
The last line follows from  Lemma \ref{lem8009.44} or Theorem \ref{thm8009.10} applied to the correlation function. Lastly, use the readjustment
\begin{equation} \label{eq9040.228}
\frac{\varphi(p-1)}{p}
=\frac{\varphi(p-1)}{p-1}\left ( 1-\frac{1}{p}\right )
\end{equation} 
to obtain the standard form of the main term.
\end{proof}

\subsection{Main Term For $s$-Power Free Primitive Roots}
\begin{lem} \label{lem9739.06}  Let \(p\geq 2\) be a large prime, let $s\geq 2$ be an integer, and let $\mu_s$ be the characteristic function of $s$-power free integers. Then, 
\begin{equation} \label{eq9739.08}
\sum_{ n\in \F_p}   \frac{\mu_s(n)}{p}\sum_{\gcd(m,p-1)=1} 1 =\frac{1}{\zeta(s)}\frac{\varphi(p-1)}{p-1}p+O\left (p^{1/s}\right).
\end{equation} 
\end{lem}

\begin{proof} Simplify the double sum:
\begin{equation}\label{eq8439.196}
\frac{1}{p}\sum_{ n\in \F_p} \mu_s(n) \sum_{\gcd(m,p-1)=1}1 =\frac{\varphi(p-1)}{p}\sum_{ n\in \F_p} \mu_s(n).
\end{equation}
Replace the characteristic function for $s$-power free integers, see Lemma \ref{lem297.58}, and reverse the order of summation:
\begin{eqnarray}\label{eq8439.198}
\frac{\varphi(p-1)}{p}\sum_{ n\in \F_p} \mu_s(n)
&=&\frac{\varphi(p-1)}{p}\sum_{ n\in \F_p} \sum_{d^s\mid n}\mu(d) \\
&=&\frac{\varphi(p-1)}{p}\sum_{d\leq p^{1/s}}\mu(d)\sum_{\substack{n\in \F_p\\d^s\mid n}} 1\nonumber\\
&=&\frac{\varphi(p-1)}{p}\sum_{d\leq p^{1/s}}\mu(d)\left (\frac{p}{d^s}+O(1) \right ) \nonumber\\
&=&\frac{1}{\zeta(s)}\frac{\varphi(p-1)}{p}p+O\left (p^{1/s}\right)\nonumber,
\end{eqnarray}
where $1/\zeta(s)=\sum_{n \geq 1} \mu(n)n^{-s}$ is the inverse zeta function. Now, use the readjustment
\begin{equation} \label{eq8439.110}
\frac{\varphi(p-1)}{p}
=\frac{\varphi(p-1)}{p-1}\left ( 1-\frac{1}{p}\right )
\end{equation} 
to obtain the standard form of the main term.
\end{proof}

\subsection{Main Term For $s$-Power Free Twin Primitive Roots}
\begin{lem} \label{lem9715.77}  Let \(p\geq 2\) be a large prime, let $a_0\ne a_1$ and $ s \geq 2$ be small integers. Let $\mu_s$ be the $s$-power free characteristic function. Then, 
\begin{eqnarray} \label{eq9715.08}
&&\sum_{ n\in \F_p}   \left (\frac{\mu_s(n+a_0)}{p}\sum_{\gcd(n_0,p-1)=1} 1 \right )  \left (\frac{\mu_s(n+a_1)}{p}\sum_{\gcd(n_1,p-1)=1} 1 \right ) \\
&=&\prod_{q \geq 2}\left (1-\frac{\rho(s)}{q^s}\right )\left (\frac{\varphi(p-1)}{p-1} \right )^2p+ O\left ( p^{\alpha(s)+\varepsilon}\right ) \nonumber,
\end{eqnarray} 
where $\rho(s)=1,2$, $\alpha(s)<1$ and $\varepsilon>0$ is an arbitrary small number.
\end{lem}

\begin{proof}Rearrange it and simplify it as
\begin{eqnarray} \label{eq9715.020}
M_2(2,p)
&=&\sum_{ n\in \F_p}   \left (\frac{\mu_s(n)}{p}\sum_{\gcd(n_0,p-1)=1} 1 \right )  \left (\frac{\mu_s(n+a)}{p}\sum_{\gcd(n_1,p-1)=1} 1 \right ) \nonumber\\
&=&\left (\frac{\varphi(p-1)}{p} \right )^2\sum_{ n\in \F_p}  \mu_s(n)\mu_s(n+a) \\
&=&\left (\frac{\varphi(p-1)}{p} \right )^2\left (\prod_{q \geq 2}\left (1-\frac{\rho(s)}{q^s}\right )p+ O\left ( p^{\alpha(s)+\varepsilon}\right )\right )\nonumber
\end{eqnarray}
The last line follows from Theorem \ref{thm9549.310} applied to the correlation function. Lastly, use the readjustment
\begin{equation} \label{eq9715.022}
\frac{\varphi(p-1)}{p}
=\frac{\varphi(p-1)}{p-1}\left ( 1-\frac{1}{p}\right )
\end{equation} 
to obtain the standard form of the main term.
\end{proof}
\subsection{Main Term For Relatively Prime Primitive Roots}
 
\begin{lem} \label{lem9729.256}  Let \(p\geq 2\) be a large prime, and let $q \leq p-1$ be a fixed integer. Then, 
\begin{equation} \label{eq9729.208}
\frac{1}{p}\sum_{\substack{ n\in \F_p\\ \gcd(n,q)=1}} \sum_{\gcd(m,p-1)=1}1
=\frac{\varphi(q)}{q}\frac{\varphi(p-1)}{p-1}p +O(\log^2 p).
\end{equation} 
\end{lem}

\begin{proof} Simplify the double sum:
\begin{eqnarray}\label{eq9729.218}
M_r(p,q)&=&\frac{1}{p}\sum_{\substack{ n\in \F_p\\ \gcd(n,q)=1}} \sum_{\gcd(m,p-1)=1}1 \\
&=&\frac{\varphi(p-1)}{p}\sum_{\substack{ n\in \F_p\\ \gcd(n,q)=1}} 1 \nonumber .
\end{eqnarray}
Replace the characteristic function for relatively prime numbers, see Definition \ref{dfn297.35}, and rearrange the order of summation:
\begin{eqnarray}\label{eq9729.225}
\frac{\varphi(p-1)}{p}\sum_{\substack{ n\in \F_p\\ \gcd(n,q)=1}} 1 
&=&\frac{\varphi(p-1)}{p}\sum_{n\in \F_p} \sum_{\substack{d\mid n\\ d\mid q}} \mu(d)\\
&=&\frac{\varphi(p-1)}{p}\sum_{d\mid q}\mu(d)\sum_{\substack{n\in \F_p\\d\mid n}} 1\nonumber\\
&=&p\frac{\varphi(p-1)}{p}\sum_{d\mid q}\frac{\mu(d)}{d}\nonumber\\
&=&\frac{\varphi(q)}{q}\frac{\varphi(p-1)}{p-1}p \nonumber,
\end{eqnarray}
where $\varphi(n)/n=\sum_{ d\mid n}\mu(d)/d$, see Section \ref{s8533}. Lastly, use the readjustment
\begin{equation} \label{eq9729.228}
\frac{\varphi(p-1)}{p}
=\frac{\varphi(p-1)}{p-1}\left ( 1-\frac{1}{p}\right )
\end{equation} 
to obtain the standard form of the main term.
\end{proof}
\subsection{Main Term For Relatively Prime Twin Primitive Roots}
The identity $\varphi(n)=\sum_{\gcd(d,n)=1}1$, and the estimate $\sum_{d\mid q}|\mu(d)|=O\left (q^{\delta}\right ) $ for $\delta>0$ is a small number, see Section \ref{s8533}, are used within the proofs.

\begin{lem} \label{lem9729.356}  If \(p\geq 2\) is a large prime, let $a\geq 1$ and $q \leq p-1$ be a pair of fixed integers. Then, 
\begin{equation} \label{eq9729.308}
\frac{1}{p}\sum_{\substack{ n\in \F_p\\ \gcd(n,q)=1\\ \gcd(n+a,q)=1}} \sum_{\gcd(m,p-1)=1}1
=c_2(q,a)\left (\frac{\varphi(q)}{q}\right )^2\frac{\varphi(p-1)}{p-1}p +O\left (p^{2\delta}\right ),
\end{equation} 
where $c_2(q,a)\geq 0$ is a dependence correction factor, and $\delta>0$ is a small number.
\end{lem}

\begin{proof} Simplify the double sum:
\begin{equation}\label{eq9729.318}
\frac{1}{p}\sum_{\substack{ n\in \F_p\\ \gcd(n,q)=1\\ \gcd(n+a,q)=1}} \sum_{\gcd(m,p-1)=1}1=\frac{\varphi(p-1)}{p}\sum_{\substack{ n\in \F_p\\ \gcd(n,q)=1\\ \gcd(n+a,q)=1}} 1 .
\end{equation}
Replace the characteristic function for relatively prime numbers, see Definition \ref{dfn297.35}, and rearrange the order of summation:
\begin{eqnarray}\label{eq9729.325}
\frac{\varphi(p-1)}{p}\sum_{\substack{ n\in \F_p\\ \gcd(n,q)=1\\ \gcd(n+a,q)=1}} 1
&=&\frac{\varphi(p-1)}{p}\sum_{n\in \F_p} \sum_{\substack{d\mid n\\ d\mid q}} \mu(d) \sum_{\substack{e\mid n+a\\ e\mid q}} \mu(e)\\
&=&\frac{\varphi(p-1)}{p}\sum_{d\mid q}\mu(d)\sum_{e\mid q}\mu(e)\sum_{\substack{n\in \F_p\\d\mid n\\e\mid n+1}} 1\nonumber\\
&=&\frac{\varphi(p-1)}{p}\sum_{d\mid q}\mu(d)\sum_{e\mid q}\mu(e)\left ( c_2(q,a)\frac{p}{de}+O(1) \right )\nonumber,
\end{eqnarray}
where $c_2(q,a)\geq0$ is a dependence correction factor. Continuing yield
\begin{eqnarray}\label{eq9729.327}
\frac{\varphi(p-1)}{p}\sum_{\substack{ n\in \F_p\\ \gcd(n,q)=1\\ \gcd(n+a,q)=1}} 1
&=& c_2(q,a)p\frac{\varphi(p-1)}{p}\sum_{d\mid q}\frac{\mu(d)}{d}\sum_{e\mid q}\frac{\mu(e)}{e}+O\left (\sum_{d\mid q}|\mu(d)|\sum_{e\mid q}|\mu(e)| \right )\nonumber\\
&=& c_2(q,a)\left (\frac{\varphi(q)}{q}\right )^2\frac{\varphi(p-1)}{p}p +O\left (q^{2\delta}\right ) ,
\end{eqnarray}
where $\delta>0$ is a small number, and $\sum_{d\mid q}|\mu(d)|=O(q^{\delta})=O\left (p^{\delta}\right )$.
Lastly, use the readjustment
\begin{equation} \label{eq9729.328}
\frac{\varphi(p-1)}{p}
=\frac{\varphi(p-1)}{p-1}\left ( 1-\frac{1}{p}\right )
\end{equation} 
to obtain the standard form of the main term.
\end{proof}

The above proof is a simplified version, it does not show the details of the dependence between the variables $d\mid$ and $e \mid q$ in the last line of \eqref{eq9729.325}. It simply includes a dependence correction constant $c_2(q)>0$.

\subsection{Main Term For Squarefree And Relatively Prime Primitive Roots} \label{lem9009.256}
\begin{lem} \label{lem9009.421}  Let \(p\geq 2\) be a large prime, and let $q=O(\log p) $ be a fixed integer. Then, 
\begin{equation} \label{eq9009.208}
\frac{1}{p}\sum_{\substack{ n\in \F_p\\ \gcd(n,q)=1}} \sum_{\gcd(m,p-1)=1}\mu(n)^2
=\frac{6}{\pi^2}\prod_{p\nmid q}\left ( 1+\frac{1}{p}\right )^{-1} \frac{\varphi(p-1)}{p-1}p +O \left (p^{1/2} \right )
\end{equation} 
\end{lem}
\begin{proof} Simplify the double sum:
\begin{eqnarray}\label{eq9009.218}
M_r(p,q)&=&\frac{1}{p}\sum_{\substack{ n\in \F_p\\ \gcd(n,q)=1}} \sum_{\gcd(m,p-1)=1}\mu(n)^2 \\
&=&\frac{\varphi(p-1)}{p}\sum_{\substack{ n\in \F_p\\ \gcd(n,q)=1}} \mu(n)^2 \nonumber .
\end{eqnarray}
Apply Lemma \ref{lem9339.117} to the inner sum:
\begin{eqnarray}\label{eq9009.225}
\frac{\varphi(p-1)}{p}\sum_{\substack{ n\in \F_p\\ \gcd(n,q)=1}} \mu(n)^2
&=&\frac{\varphi(p-1)}{p}\left (\frac{6}{\pi^2}\prod_{p\nmid q}\left ( 1+\frac{1}{p}\right )^{-1} p +O \left (p^{1/2} \right ) \right ) \nonumber\\
&=&\frac{6}{\pi^2}\prod_{p\nmid q}\left ( 1+\frac{1}{p}\right )^{-1}\frac{\varphi(p-1)}{p}p +O \left (p^{1/2} \right ).
\end{eqnarray}
Lastly, use the readjustment
\begin{equation} \label{eq9009.228}
\frac{\varphi(p-1)}{p}
=\frac{\varphi(p-1)}{p-1}\left ( 1-\frac{1}{p}\right )
\end{equation} 
to obtain the standard form of the main term.
\end{proof}


\subsection{Main Term For Squarefree And Relatively Prime Twin Primitive Roots} \label{ss9792}

\begin{lem} \label{lem9792.356}  Assume conjecture {\normalfont \ref{conj8009.105}}. If \(p\geq 2\) is a large prime, let $a\geq 1$ and $q \leq p-1$ be a pair of fixed integers. Then, 
\begin{eqnarray} \label{eq9792.308}
&& \frac{1}{p^2}\sum_{\substack{ n\in \F_p\\ \gcd(n,q)=1\\ \gcd(n+a,q)=1}} \sum_{\substack{\gcd(m_0,p-1)=1\\\gcd(m_1,p-1)=1}}\mu(n)^2\mu(n+a)^2\\
&=&c_2(q,a)\prod_{r\mid q}\left ( 1-\frac{1}{r^s}\right )\prod_{p\geq 2}\left ( 1-\frac{2}{p^s}\right )
\left (\frac{\varphi(p-1)}{p-1}\right )^2p +O\left (p^{2\delta}\right ) \nonumber,
\end{eqnarray} 
where $c_2(q,a)\geq0$ is a dependence correction factor, and $\delta>0$ is a small number.
\end{lem}

\begin{proof} Simplify the double sum:
\begin{equation}\label{eq9792.318}
\frac{1}{p^2}\sum_{\substack{ n\in \F_p\\ \gcd(n,q)=1\\ \gcd(n+a,q)=1}} \sum_{\substack{\gcd(m_0,p-1)=1\\\gcd(m_1,p-1)=1}}\mu(n)^2\mu(n+a)^2=\left (\frac{\varphi(p-1)}{p}\right )^2\sum_{\substack{ n\in \F_p\\ \gcd(n,q)=1\\ \gcd(n+a,q)=1}}\mu(n)^2\mu(n+a)^2.
\end{equation}
Set $x=p$, and apply Conjecture \ref{conj8009.105} to the inner finite sum:
\begin{eqnarray}\label{eq9792.325}
M_{sr}(2,p,q)&=&\left (\frac{\varphi(p-1)}{p}\right )^2\sum_{\substack{ n\in \F_p\\ \gcd(n,q)=1\\ \gcd(n+a,q)=1}}\mu(n)^2\mu(n+a)^2\\
&=&\left (\frac{\varphi(p-1)}{p}\right )^2 \left (c_2(q,a)\prod_{p\nmid q}\left ( 1+\frac{1}{p}\right )^{-2}\prod_{p\geq 2}\left ( 1-\frac{2}{p^2}\right )p+O\left (p^{1-\delta} \right )\right )\nonumber\\
&=& c_2(q,a)\prod_{p\nmid q}\left ( 1+\frac{1}{p}\right )^{-2}\prod_{p\geq 2}\left ( 1-\frac{2}{p^2}\right )\left (\frac{\varphi(p-1)}{p}\right )^2p+O\left (p^{1-\delta} \right )\nonumber,
\end{eqnarray}
where $c_2(q,a)\geq0$ is a dependence correction factor, and $\delta>0$ is a small number.
Lastly, use the readjustment
\begin{equation} \label{eq9792.328}
\frac{\varphi(p-1)}{p}
=\frac{\varphi(p-1)}{p-1}\left ( 1-\frac{1}{p}\right )
\end{equation} 
to obtain the standard form of the main term.
\end{proof}

\section{The Estimates For The Error Terms}  \label{s8899}
The upper bounds for exponential sums over subsets of elements in finite fields $\mathbb{F}_p$ studied in Section \ref{s4} are used to estimate the error terms for the different configurations of consecutive primitive roots in Theorem \ref{thm8800.050} and the other results.

\begin{lem} \label{lem8899.09}  Let \(p\geq 2\) be a large prime, and let \(\tau\) be a primitive root mod \(p\). If the element \(u\ne 0, \pm1, v^2\) is not a primitive root, then, 
	\begin{equation} \label{eq8899.05}
S(p,k) =	\sum_{ u\in \F_p} \left ( \frac{1}{p}
	\sum_{\gcd(n,p-1)=1, } \sum_{0<m \leq p-1}e^{i2 \pi \left((\tau ^{n}-u)m\right)}\right )     \ll p^{1-\varepsilon}  
	\end{equation} 
	for all sufficiently large primes $p \geq 2$, and an arbitrarily small number \(\varepsilon >0\).
\end{lem}

\begin{proof}  By hypothesis \(u\ne 0, \pm1, v^2\) is not a primitive root. Thus, $S_1\ne -\varphi(p-1)$. Rearrange the finite sum as 
\begin{eqnarray} \label{eq8899.05}
S_1&=&\sum_{ u\in \F_p} \frac{1}{p}
	\sum_{\gcd(n,p-1)=1, } \sum_{0<m \leq p-1}e^{i2 \pi \left((\tau ^{n}-u)m\right)} \\  
&= & \frac{1}{p}\sum_{ u\in \F_p}  \left (\sum_{ 0<m\leq p-1,} e^{-i 2 \pi um/p} \right ) \left ( \sum_{\gcd(n,p-1)=1} e^{i 2 \pi m\tau ^n/p} \right )\nonumber \\
 &= & \frac{1}{p}\sum_{ u\in \F_p} \left (\sum_{ 0<m\leq p-1,} e^{-i 2 \pi um/p} \right ) \left ( \sum_{\gcd(n, p-1)=1} e^{i2\pi  \tau^{n}/p} + O(p^{1/2} \log^3 p) \right )\nonumber \\
 &= & \frac{1}{p}\sum_{ u\in \F_p}  U_p \cdot V_p \nonumber.
\end{eqnarray} 
The third line in equation (\ref{eq8899.05}) follows from Lemma \ref{lem333.22}. The first exponential sum $U_p$ has the exact evaluation
\begin{equation}\label{eq8899.13}
| U_p| = \left |\sum_{ 0<m\leq p-1} e^{-i 2 \pi um/p} \right |=1,
\end{equation} 
where $\sum_{ 0<m\leq p-1} e^{i 2 \pi um/p}=-1$ for any $u \in [1,p-1]$. The second exponential sum $V_p$ has the upper bound
\begin{eqnarray} \label{eq8899.15}
|V_p|&=& \left |\sum_{\gcd(n,p-1)=1} e^{i2 \pi \tau ^n/p}+ O\left (p^{1/2} \log^3 p \right ) \right |\nonumber \\
&\ll &\left |\sum_{\gcd(n,p-1)=1} e^{i2 \pi \tau ^n/p} \right |+p^{1/2} \log^3 p  \\
&\ll&  p^{1-\varepsilon} \nonumber,
\end{eqnarray} 
where \(\varepsilon <1/2 \) is an arbitrarily small number, see Theorem \ref{thm333.04}. Taking absolute value in (\ref{eq8899.05}), and replacing the estimates (\ref{eq8899.13}) and (\ref{eq8899.15}) return
\begin{eqnarray} \label{eq8899.20}
\left |S_1 \right | &\leq &  \frac{1}{p}\sum_{u\in \F_p}  \left | U_p \right | \cdot  |V_p| \\
&\ll &\frac{1}{p}\sum_{ u\in \F_p}  (1) \cdot    p^{1-\varepsilon } \nonumber\\
&\ll &  \frac{ 1}{p^{\varepsilon }}\sum_{ u\in \F_p}  1 \nonumber \\
&\ll & p^{1-\varepsilon}\nonumber.
\end{eqnarray}
\end{proof}

No effort was made to optimize the error term in Lemma \ref{lem8899.09}. However, it should be noted that the best possible is $ p^{1/2+\varepsilon}  $, see Theorem \ref{thm333.04}.
\subsection{Error Term For $k+1$ Consecutive Primitive Roots}
\begin{lem} \label{lem8899.06}  Let \(p\geq 2\) be a large prime, let \(k<  \log p/\log \log \log p\) be an integer, and let \(\tau\) be a primitive root mod \(p\). If the element \(n+a_i\ne 0, \pm1, v^2\) is not a primitive root for $i=0,1,2,...,k$, then, 
	\begin{equation} \label{eq8899.00}
E(k,p) =	\sum_{ n\in \F_p}\prod_{0 \leq i\leq k} \left ( \frac{1}{p}
	\sum_{\substack{\gcd(n_i,p-1)=1 \\ 0<m_i \leq p-1}} e^{i2 \pi \left((\tau ^{n_i}-n-a_i)m_i\right)}\right )       \ll p^{1-\varepsilon}  
	\end{equation} 
	for all sufficiently large primes $p \geq 2$, and an arbitrarily small number \(\varepsilon >0\).
\end{lem}

\begin{proof}  By hypothesis $n+a_i\ne 0, \pm1, v^2$ is not a primitive root for $i=0,1,2,...,k$. Thus, $E(k,p)\ne -(\varphi(p-1)/p)^{k+1}p$. Rewrite the multiple finite sum as a product $E(p,\tau) =S_1 \times S_2$. The first sum indexed by $m=m_0$ and $n=n_0$ has a nontrivial upper bound
\begin{equation} \label{eq8899.20}
\left |S_1 \right | \ll  p^{1-\varepsilon},
\end{equation}
see Lemma \ref{lem8899.09}. The product of the remaining sums indexed by $m_i$ and $n_i$, $i\in \{1, 2, \ldots  k-1\}$ have the trivial upper bound
\begin{eqnarray} \label{eq8899.25}
\left |S_2 \right | &\leq &  \left |\frac{1}{p} \sum_{\substack{\gcd(n_1,p-1)=1 \\ 0<m_1 \leq p-1}} e^{i2 \pi \left((\tau ^{n_1}-n-a_1)m_1\right)}  \right | 
 \cdots  
\left | \frac{1}{p}\sum_{\substack{\gcd(n_{k},p-1)=1 \\ 0<m_{k} \leq p-1}} e^{i2 \pi \left((\tau ^{n_{k}}-n-a_k)m_{k}\right)} \right |  \nonumber \\
&\leq &  \frac{\varphi(p-1)}{p}  
\cdots  \frac{\varphi(p-1)}{p} \\
&\leq &\left (\frac{\varphi(p-1)}{p} \right )^{k}\nonumber.
\end{eqnarray}
Merging \eqref{eq8899.20} and \eqref{eq8899.25} returns
\begin{eqnarray} \label{eq8899.30}
\left |E(p,\tau) \right |&\leq  & \left |S_1 \right |\left |S_2 \right | \\
&\leq &  \left ( p^{1-\varepsilon} \right) \times \left (\frac{\varphi(p-1)}{p} \right )^{k}\nonumber\\
&\leq& p^{1-\varepsilon}\nonumber.
\end{eqnarray}
The last inequality uses $\varphi(p-1)/p\leq 1$.
\end{proof}

\subsection{Error Term For $s$-Power Free Primitive Roots}
\begin{lem} \label{lem8809.02}  Let \(p\geq 2\) be a large prime, let \(\tau\) be a primitive root mod \(p\), and let $\mu_s$ be the characteristic function of $s$-power free integers. If the element \(n\ne 0, \pm1, v^2\) is not a primitive root, then,   
\begin{equation} \label{eq8809.00}
E(s,p) =
\sum_{ n\in \F_p}   \left (\frac{\mu_s(n)}{p}\sum_{\substack{\gcd(m,p-1)=1\\
1\leq a\leq p-1}} \psi \left((\tau ^{m}-n)a\right) \right )    
 \ll p^{1-\varepsilon}  
\end{equation} 
for all sufficiently large primes $p \geq 2$, and an arbitrarily small number \(\varepsilon >0\).
\end{lem}

\begin{proof}  Same as Lemma \ref{lem8899.09}, mutatis mutandus.
\end{proof}

\subsection{Error Term For $k+1$ Consecutive Squarefree Primitive Roots}
\begin{lem} \label{lem8829.08}  Let \(p\geq 2\) be a large prime, let  $0 \leq a_0, a_1, a_2, \ldots,a_k$ be an admissible $(k+1)$-tuple of integers, and let \(\tau\) be a primitive root modulo $p$. If the element \(n+a_i\ne 0, \pm1, v^2\) is not a primitive root for $i=0,1,2,...,k$, then,   
\begin{equation} \label{eq8829.08}
E_s(k,p) =
\sum_{ n\in \F_p}   \prod_{0 \leq i\leq k} \left (\frac{\mu^{2}(n+a_i)}{p}\sum_{\substack{\gcd(n_i,p-1)=1\\
1\leq b_i\leq p-1}} \psi \left((\tau ^{n_i}-n-a_i)b_i\right) \right ) 
 \ll p^{1-\varepsilon}  
\end{equation} 
for all sufficiently large primes $p \geq 2$, and an arbitrarily small number \(\varepsilon >0\).
\end{lem}

\begin{proof}  Same as Lemma \ref{lem8899.06}, mutatis mutandus.
\end{proof}

\subsection{Error Term For Restricted $k+1$ Consecutive Primitive Roots}
\begin{lem} \label{lem8829.108}  Let \(p\geq 2\) be a large prime, let  $0 \leq a_0, a_1, a_2, \ldots,a_k$ be an admissible $(k+1)$-tuple of integers, and let \(\tau\) be a primitive root modulo $p$. If the element \(n+a_i\ne 0, \pm1, v^2\) is not a primitive root for $i=0,1,2,...,k$, and $f(n)\ll 1$ is a bounded arithmetic function, then,   
\begin{equation} \label{eq8829.08}
E_s(k,p) =
\sum_{ n\in \F_p}   \prod_{0 \leq i\leq k} \left (\frac{f(n+a_i)}{p}\sum_{\substack{\gcd(n_i,p-1)=1\\
1\leq b_i\leq p-1}} \psi \left((\tau ^{n_i}-n-a_i)b_i\right) \right ) 
 \ll p^{1-\varepsilon}  
\end{equation} 
for all sufficiently large primes $p \geq 2$, and an arbitrarily small number \(\varepsilon >0\).
\end{lem}

\begin{proof}  Use the fact that $|f(n)|\ll 1$, and the same technique as Lemma \ref{lem8899.06}, mutatis mutandus.
\end{proof}

\section{Some Collections Of Primes} \label{s9292}
Some information on the collections of primes of interest in the theory of consecutive and quasi consecutive primitive roots are recorded here.

\subsection{Gauss Probabilistic Method}
The basic Gauss probabilistic prime counting method for the number of primes generated by an arithmetic function $f: \mathbb{N} \longrightarrow \mathbb{N}$ has the shape
\begin{equation}\label{eq333333}
\pi_f(x)=\#\{p=f(n)\leq x\}=\sum_{f(n)\leq x} \frac{1}{\log f(n)}+O\left ( \frac{x^{1/d}}{(\log x)^2}\right ),
\end{equation}	
where $f$ is a polynomial of degree $d=\deg f$. The conversion to an Stiejles integral realizes the equivalent asymptotic formulas
\begin{equation}
\pi_f(x)=\int_2^x\frac{1}{(\log f(t)}dt+O\left(\frac{x^{1/d}}{\log^{k+1} x} \right),
\end{equation}
and
\begin{equation}
\sum_{f_i(n) \leq x} \Lambda(f_1(n)) \cdots \Lambda(f_k(n))=s_f x^{1/d} +O\left(\frac{x^{1/d}}{\log^C x} \right) \nonumber ,
\end{equation}
with $d=d_1+d_2+\cdots+d_k$, and $C>0$ is arbitrary.
The Bateman-Horn conjecture is a refined version geared for a product $f(x)=f_1(x)\cdots f_k(x)\in  \mathbb{Z}[x]$ of $k\geq 1$ irreducible polynomials $f_i(x)$ of degree $d_i=\deg f_i$.

\subsection{Polynomials Primes Values Conjecture}
The quantitative form of the qualitative Hypothesis H, \cite[pp. 386--394]{RP96}, was formulated about fifty years ago in \cite{BH62}. It states the followings.	
\begin{conj} \label{conj2.1} {\normalfont(Bateman-Horn)} \label{conj1.1}   Let $f_1(x), f_2(x), \ldots, f_k(x) \in \mathbb{Z}[x]$ be relatively prime polynomials of degree $\deg(f_i)=d_i \geq 1$. Suppose that each polynomial $f_i(x)$ has the fixed divisor $\tdiv(f_i) = 1$. Then, the number of simultaneously primes $k$-tuples
\begin{equation}
f_1(n),f_2(n),\ldots,f_k(n), \nonumber
\end{equation}
as $n \leq x$ tends to infinity has the equivalent asymptotic formulas
\begin{equation}
\pi_f(x)=s_f \int_2^x\frac{1}{(\log t)^k}dt+O\left(\frac{x^{1/d}}{\log^{k+1} x} \right),
\end{equation}
and
\begin{equation}
\sum_{f_i(n) \leq x} \Lambda(f_1(n)) \cdots \Lambda(f_k(n))=s_f x^{1/d} +O\left(\frac{x^{1/d}}{\log^C x} \right) \nonumber ,
\end{equation}
with $d=d_1+d_2+\cdots+d_k$, and $C>0$ is arbitrary. The density constant is defined by the product
\begin{equation}
s_f=\frac{1}{d_1d_2 \cdots d_k}\prod_{p\geq 2} \left (1-\frac{v_p(f)}{p} \right ) \left (1-\frac{1}{p} \right )^{-k},
\end{equation} 
where the symbol $v_p(f)\geq 0$ denotes the number of solutions of the congruence
\begin{equation}
f_1(x)f_2(x) \cdots f_k(x) \equiv 0 \tmod p.
\end{equation}
\end{conj}
This generalization of prime values of polynomials appears in \cite{BH62}. The constant, known as singular series
\begin{equation} 
\mathfrak{G}(f)=\prod_{p\geq 2} \left (1-\frac{v_p(f)}{p} \right ) \left (1-\frac{1}{p} \right )^{-k},
\end{equation}
can be derived by either the circle method, as done in \cite[p.\ 167]{IK04}, \cite{VR73} or by probabilistic means as explained in \cite[p.\ 33]{PJ09}, 
\cite[p.\ 410]{RP96}, et alii. The convergence of the product is discussed in \cite{BH62}, \cite{RI15}, and other by authors. \\

The term 
\begin{equation}\label{eq333373}
P_f(n)=\prod_{p\leq n} \left ( 1-\frac{\nu(p)}{p} \right )\left ( 1-\frac{1}{p} \right )^{-k} ,
\end{equation}
where $\nu(p)=\#\{n:f(n)\equiv 0 \bmod p\}$, is a correction factor accounting for some of the small primes dependence.\\

The Bateman-Horn conjecture was formulated in the 1960's, about half a century ago. A more recent result accounts for the irregularities and oscillations that can occur in the asymptotic formula as the polynomial is varied.	
\begin{thm} \label{thm2.2} {\normalfont (\cite{FG91})}  Let $d \geq 1$ be a fixed integer, and let $B \geq 2$ be a real number. There exist infinitely many irreducible polynomials $f(x)$ of degree $\deg(f)=d$ with nonnegative integer coefficients, such that for some number $\delta_B>0$ depending on $B$, and $x \geq \log^B |f(x)|$, the absolute difference
\begin{equation}
\left | \pi_f(x) -  C_f \frac{x}{\log |f(x)|} \right |>\delta_B C_f \frac{x}{\log |f(x)|}.
\end{equation}
\end{thm}
This result was proved for polynomials of large degrees, but it is claimed to hold for certain structured polynomials of small degrees $d \geq 1$.

\subsection{Primorial Primes}\label{s9292B}
The subset of primorial primes 
\begin{equation}  \label{eq9292.102}
\mathcal{A}=\{p=2\cdot3\cdot 5\cdot 7\cdots q+1: \text{ prime } q \geq 2\}
\end{equation}
is studied in \cite{CG02}, and listed in \text{OEIS} A014545. A primorial prime has a highly composite totient
$p-1=\varphi(p)$, the maximal numbers of prime divisors
\begin{equation} \label{eq9292.170}
\omega(p-1)\ll \log p/\log \log p,
\end{equation} 
see Lemma \ref{lem8533.05}, and the minimal value 
\begin{equation} \label{eq9292.173}
\frac{\varphi(p-1)}{p-1} =\prod_{q\mid p-1}\left (1-\frac{1}{q}\right )\approx \frac{1}{\log \log p},
\end{equation} 
where $r\leq q$ ranges over the primes, see Theorem \ref{thm9192.21}. The heuristic claims that there are infinitely many primorial primes. The theory of the subset of primes is at a rudimentary stage, and a topic of current research. \\

The standard heuristic for the number of primorial primes is based on the gaussian probabilistic method.
\begin{conj} \label{conj9292.109} As $x \to infty$, the number of prime pairs $p=n!+1\ leq x$ has the asymptotic formula
\begin{equation}\label{eq9292.109}
\pi_f(x)=\#\{p=n!+1\leq x: p \text{ is prime } \}= e^{\gamma}\log x +o(\log x).\nonumber
\end{equation}
\end{conj}
\textbf{Heuristic 1:} As $n \to \infty$, the asymptotic for the factorial function is $n!\approx n \log n$. Thus, 
\begin{eqnarray}\label{eq9292.107}
\pi_f(x)&=&\#\{p=n!+1\leq x: p \text{ is prime } \}\\
&=&\sum_{n\leq x} \prod_{p\leq n}  \left ( 1-\frac{1}{p} \right ) ^{-1}\frac{1}{\log (n!+1)}\nonumber \\
&\approx& \sum_{n\leq x} e^{\gamma}\log n \cdot \frac{1}{n \log n}\nonumber\\
&=& e^{\gamma}\log x +o(\log x)\nonumber.
\end{eqnarray}
It is not a standard practice, but a heuristic based on the Bate-Horn conjecture seems to predict an infinite number of primorial primes. To see this observe that the primorial primes are generated by the expression $f(n)=n!\pm1$. Here, the factorial can be viewed as a product of $n$ linear polynomials $f_k(x)=x-k$, actually this idea is used in $p$-adic analysis. Under this assumption, the Bate-Horn conjecture is applicable.\\

\textbf{Heuristic 2:} Let $f(x)=f_1(x)f_2(x)\cdots =x(x-1)(x-2)\cdots 2\cdot 1+1$. The congruence $f(n)\equiv 0 \bmod p$ has $\nu(p)=0$ solutions for all primes $p\geq 2$. Hence,
\begin{eqnarray}\label{eq334777}
\pi_f(x)&=&\#\{p=f(n)\leq x: p \text{ is prime } \}\\
&\leq&\sum_{p\leq x} \prod_{q\leq \log p}  \left ( 1-\frac{1}{q} \right )^{-n}\left ( 1-\frac{\nu(p)}{p} \right ) \frac{1}{\log (n!+1)}\nonumber \\
&=&\sum_{n\leq x} \prod_{p\leq n}  \left ( 1-\frac{1}{p} \right )^{-n} \frac{1}{\log (n!+1)}\nonumber \\
&\approx& \sum_{n\leq x} \left ( \frac{1}{e^{\gamma}\log n} \right )^{-n}  \frac{1}{n \log n}\nonumber\\
&=& O\left ( \frac{\log x}{\log \log x} \right )\nonumber.
\end{eqnarray}
Thus, the expected number is infinite.

\subsection{Coprimorial Primes}\label{s9292C}
The subset of coprimorial primes is defined by
\begin{equation}  \label{eq9292.204}
\mathcal{B}=\{p=3\cdot 5\cdot 7\cdots q+2: \text{ prime } q \geq 2\}.
\end{equation}
The totient $p-1=\varphi(p)$ of a coprimorial prime has very few prime divisors
\begin{equation} \label{eq9292.270}
\omega(p-1)\ll 1,
\end{equation} 
and nearly maximal value 
\begin{equation} \label{eq9799.279}
\frac{\varphi(p-1)}{p-1} =\prod_{q\mid p-1}\left (1-\frac{1}{q}\right )\approx \frac{1}{2}.
\end{equation} 
The coprimorial primes have Germain primes type structure. The heuristic seems to show the existence of infinitely many coprimorial primes. 

\begin{conj} \label{conj9292.049} As $x \to \infty$, the number of prime pairs $p=n!+1\ leq x$ has the asymptotic formula
\begin{equation}\label{eq9292.050}
\pi_f(x)=\#\{p=n!/2+2\leq x: p \text{ is prime } \}= e^{\gamma}\log x +o(\log x).\nonumber
\end{equation}
\end{conj}
\textbf{Heuristic 1:} Same as \ref{conj9292.109}.

The collection of these primes is not a topic of research in the current literature. 

\subsection{Germain Primes}\label{s9292D} 
The subset of Germain primes is defined by
\begin{equation}  \label{eq9292.306}
\mathcal{S}=\{p=2^a\cdot q+1: \text{ prime } q \geq 2 \text{ and  } a\geq 1\}.
\end{equation}
The simplest sequence $p=2q+1$ is archived in \text{OEIS} A005384. The heuristic claims that there are infinitely many Germain primes. The theory of the subset of Germain primes is not fully developed, but it is a topic of current research. 
The totient $p-1=\varphi(p)$ of a Germain prime has two prime divisors
\begin{equation} \label{eq9292.370}
\omega(p-1)=2,
\end{equation} 
and the nearly maximal value 
\begin{equation} \label{eq9292.375}
\frac{\varphi(p-1)}{p-1} = \prod_{q\mid p-1}\left (1-\frac{1}{q}\right )\approx \frac{1}{2},
\end{equation} 
where $r\leq q$ ranges over the primes.\\

The expected number of Germain primes is derived from the Bate-Horn conjecture using the polynomials $f_1(x)=x$ and $f_2(x)=2^ax+1$.

\begin{conj}  \label{conj9292.309} As $x \to \infty$, the number of prime pairs $p\leq x$ and $2^ap+1\leq x$ has the asymptotic formula
\begin{equation}\label{eq9292.373}
\pi_f(x)=2\prod_{p\geq 3}  \left ( 1-\frac{1}{(p-1)^2} \right )  \frac{x}{(\log x)^2}+O\left ( \frac{x}{(\log x)^3}\right ).\nonumber
\end{equation}
\end{conj}
\textbf{Heuristic:} Let $f(x)=f_1(x)f_2(x)=x(2^ax+1)$, with a small fixed integer $a\geq 1$. The congruence $f(n)\equiv 0 \bmod p$ has $\nu(p)$ solutions. Specifically, 
\begin{equation}\label{eq9292.355}
\nu(p)=
\begin{cases}
1 & \text{ if } p=2,\\
2& \text{ if } p>2.
\end{cases}
\end{equation}
Assembling these data yield
\begin{eqnarray}\label{eq9292.387}
\pi_f(x)&=&\#\{p\leq x: p \text{ and } 2^ap+1 \text{ are primes } \}\\
&=&2\sum_{n\leq x} \prod_{2<p\leq x}  \left ( 1-\frac{1}{p} \right )^{-2}\left ( 1-\frac{2}{p} \right ) \frac{1}{\log (n)\log (2^an+1)}\nonumber \\
&=& 2\prod_{2<p\leq n}  \left ( 1-\frac{1}{(p-1)^2} \right )\sum_{n\leq x}  \frac{1}{\log (n)\log (2^an+1)}\nonumber\\
&=& 2\prod_{p\geq 3}  \left ( 1-\frac{1}{(p-1)^2} \right )  \frac{x}{(\log x)^2}+O\left ( \frac{x}{(\log x)^3}\right )\nonumber.
\end{eqnarray}
The last line in \eqref{eq9292.387} follows from the approximation
\begin{equation}\label{eq9292.82}
\sum_{n\leq x}  \frac{1}{\log (n)\log (2^an+1)}\approx \sum_{n\leq x}  \frac{1}{\log (n)^2}=\int_2^x\frac{1}{\log (t)^2}dt.
\end{equation}

\subsection{Fermat Primes}\label{s9292E} 
The subset of Fermat primes is defined by
\begin{equation}  \label{eq9292.408}
\mathcal{S}=\{p=2^{2^n}+1:  n\geq \}.
\end{equation}
The entire list of known has 5 primes, which are archived in \text{OEIS} A019434. The heuristic claims that there are finitely many Fermat primes. The totient $p-1=\varphi(p)$ of a Fermat prime has one prime divisor
\begin{equation} \label{eq9799.480}
\omega(p-1)=1,
\end{equation} 
and the maximal value 
\begin{equation} \label{eq9799.477}
\frac{\varphi(p-1)}{p-1} =\prod_{q\mid p-1}\left (1-\frac{1}{q}\right )=\frac{1}{2}.
\end{equation} 
where $r\leq q$ ranges over the primes.

\begin{conj}  \label{conj9292.409} As $x \to \infty$, the number of primes $p=2^{2^n}+1\leq x$ is finite.
\end{conj}

\textbf{Heuristic:} Let $p=f(n)=2^{2^n}+1$, and $x \geq 1$ be a large number. The calculation is based on the Gaussian probabilistic method. The Bate-Horn conjecture is not applicable to subset of primes generated by exponential functions. Specifically, 
\begin{eqnarray}\label{eq9292.477}
\pi_f(x)&=&\#\{p=f(n)\leq x: p \text{ is prime } \}\\
&=&\sum_{f(n)\leq x} \prod_{2<p\leq n}  \left ( 1-\frac{1}{p} \right )^{-1} \frac{1}{\log \left (2^{2^n}+1\right )}\nonumber \\
&\leq &\frac{e^{\gamma}}{ \log \left (2\right )}\sum_{n\leq \log \log x} \frac{\log n}{2^n }\nonumber \\
&=& O(1)\nonumber.
\end{eqnarray}
Thus, the expected number is finite, \cite[p.\ 16]{HW79}.

\begin{lem}\label{lem0099.22} {\normalfont(Pepin test)} The number $p=2^{2^n}+1$ is prime if and only if the Legendre symbol
\begin{equation}
\left ( \frac{3}{p}\right ) \equiv -1 \bmod p.     \nonumber
\end{equation}
\end{lem}
\begin{proof} Use the quadratic reciprocity law.
\end{proof}

\newpage
\subsection{Problems}
\begin{exe} { \normalfont 
Determine whether or not the $26$th Fermat number $F_n=2^{2^{26}}+1$ is composite or prime. Specifically, compute the quadratic symbol
$$
\left ( \frac{3}{2^{2^{26}}+1}\right ) \equiv \pm 1 \bmod (2^{2^{26}}+1) .
$$
If $3$ is a quadratic nonresidue modulo $F_n$, then it is prime.
}
\end{exe}
\begin{exe} { \normalfont 
Use an elementary argument to show that the largest prime divisor of the $n$th Fermat number $F_n$ satisfies $P(F_n)>2^{n+2}\geq \log F_n$. 
}
\end{exe}

\begin{exe} { \normalfont 
Prove whether or not the $n$th Fermat number $F_n$ is squarefree. 
}
\end{exe}

\begin{exe} { \normalfont 
Use an elementary argument to show that the largest prime divisor of the $n$th Mersenne number $M_n=2^n-1$ satisfies $P(M_n)>n\geq \log M_n$. 
}
\end{exe}
\begin{exe} { \normalfont 
Prove whether or not the $n$th Mersenne number $M_n$ is squarefree- it is sufficient to use prime values $n=p$. 
}
\end{exe}

\begin{exe} { \normalfont 
Let $p_k$ be the $k$th prime in increasing order. Prove whether or not the subset of primorial primes $p=2\cdot3\cdots p_k\pm1$ is finite. 
}
\end{exe}

\section{Maximal Length Of Consecutive Primitive Roots} \label{s9799}
The number of prime divisors $\omega(n)$ of a random integer $n \in \N$ is a normal random variable with mean $\log \log n$, and standard error $\sqrt{\log \log n}$, see Theorem \ref{thm8533.01}, and Lemma \ref{lem8533.05}. Roughly, there are three major classes of totients $p-1=\varphi(p)$ and the corresponding classes of the primes divisors counting function $\omega(p-1)$.

\begin{enumerate} [font=\normalfont, label=(\arabic*)]
\item The subset of primorial primes $p=2\cdot3\cdot 5\cdot 7\cdots q+1$ have highly composite totients $p-1=\varphi(p)$ and the maximal numbers of prime divisors, see Subsection \ref{s9292B}. 
\item The average primes $p \geq 2$. The average totients $p-1=\varphi(p)$ have the mean numbers of prime divisors, Lemma \ref{lem8533.05}.

\item The subset of Fermat primes, Germain primes, and coprimorial primes. The totient $p-1=\varphi(p)$ of any of these primes has the minimal number of prime divisors. These primes are described in Section \ref{s9292}.  
\end{enumerate}

\begin{lem} \label{lem9799.156}  Let \(p\geq 2\) be a large prime. Then, the maximal length $k \geq 1$ of a string of consecutive primitive roots is as follows.
\begin{enumerate} [font=\normalfont, label=(\roman*)]
\item$ \displaystyle 
k \ll \log p/\log \log \log p$,  \tabto{8cm} if $\omega(p-1)\ll \log p/ \log \log p$.
\item$ \displaystyle 
k \ll \log p/\log \log \log \log p$,  \tabto{8cm}  if $\omega(p-1)\ll \log \log p$.
\item$ \displaystyle 
k \ll \log p$, \tabto{8cm} if $\omega(p-1)\ll 1$.
\end{enumerate}
 \end{lem}
\begin{proof} The existence of an $(k+1)$-tuple implies that 
\begin{equation} \label{eq9799.300}
p\left (\frac{\varphi(p-1)}{p-1} \right )^{k+1} \gg p^{1-\varepsilon}
\end{equation}
is true, with $\varepsilon \in (0,1/2)$, see Theorem \ref{thm8800.040}. Equivalently, this is
\begin{equation} \label{eq9799.304}
k \ll\frac{\varepsilon \log p}{\log \left ( \frac{p-1}{\varphi(p-1)} \right )} \ll \frac{\varepsilon \log p}{\log \omega(p-1) }.
\end{equation}
The three different cases are for
\begin{equation} \label{eq9799.304}
\frac{p-1}{\varphi(p-1)} \approx \log \log p    , \quad  \frac{p-1}{\varphi(p-1)} \approx \log \log \log p, \quad \text{ and } \quad   \frac{p-1}{\varphi(p-1)} \approx 2          
\end{equation}
respectively.
\end{proof}

\begin{dfn} \label{dfn9799.100} {\normalfont Given a prime $p\geq 2$, and $k$ the longest run of consecutive primitive roots in the finite field $\F_p$, the \textit{length merit ratio} is defined by $\hat{m}=k/\log p$. }
\end{dfn}

The length merit ratio varies as $p \to \infty$, but it remains bounded by a constant $\hat{m}\ll1$. The Fermat primes $p=2^{2^n}+1$, $n \geq 0$, the Germain primes $p=2^{a}q+1$, $q \geq 2$ primes and $a \geq 1$, and some other collections, are expected to have the largest length merit ratio. Some numerical data for small primes are provided here. Observe that these small cases are subject to the Strong law of small numbers, \cite{GR88}. 

\begin{exa} \label{exa9799.010} {\normalfont Extreme Case 1. Some statistic for the finite field $\F_{p}$ with $p=p=2^{4}+1=17$. \\
\begin{tabular}{ l | l }
\hline
 Prime & $p=17$ \\
\hline
Parameters&$\omega(p-1)=1$, $\varphi(p-1)=8$\\
\hline
  Primitive roots & $3,5,6,7,10,11,12,14$ \\
\hline
  Length $k$ & $3$ \\
\hline
Merit factor &$k/\log p=1.058869$
\end{tabular}
\vskip .25 in
Similarly, the prime $p=2^{16}+1$ has the parameters, $\omega(p-1)=1$, $\varphi(p-1)=2^{15}$, and $\log p=11.09$. Thus, Lemma \ref{lem9799.156} predicts the existence of some 11-tuples or larger $k$-tuples of consecutive primitive roots in the set of primitive roots $\mathcal{R}=\{3,5,7,11,13,15, \ldots.\}$.
}
\end{exa}
\begin{exa} \label{exa9799.014} {\normalfont Extreme Case 3. Some statistic for the finite field $\F_{p}$ with $p=2\cdot 3\cdot 5+1=31$. \\
\begin{tabular}{ l | l }
\hline
 Prime & $p=31$ \\
\hline
Parameters&$\omega(p-1)=3$, $\varphi(p-1)=8$\\
\hline
  Primitive roots & $3,11,12,13,17,21,22,24$ \\
\hline
  Length $k$ & $3$ \\
\hline
Merit factor &$k/\log p=0.873620$
\end{tabular}
\vskip .25 in
}
\end{exa}

\begin{table}[tp]
 \caption{Maximal $k$-Tuple of Primitive Roots Indexed By $p$.} \label{tb9799.04}\centering%
\begin{tabular}{ l | l|l| l | l|l| l | l|l}
\hline
$p$	&$k$		&$\hat{m}$     &$p$	&$k$		&$\hat{m}$&$p$	&$k$		&$\hat{m}$\\
\hline
$3$     &1                 &0.910239 &29	&2		&0.593948 &61	&2		&0.486514\\
5	&2		&1.242669    &31	&3		&0.873620 &67	&3		&0.713488\\
7	&1		&0.513898    &37	&4		&1.107751 &71	&3		&0.703782\\
11	&3		&1.251097    &41        &3		&0.807847 &73	&3		&0.699225\\
13	&2		&0.779742    &43	&3		&0.797617 &79	&3		&0.686585\\
17	&3		&1.058868    &47	&4		&1.038921 &83	&7		&1.584125\\
19	&3		&1.018869    &53	&5		&1.259353 &89	&6		&1.336708\\
23	&3		&0.956786    &59	&5		&1.226230 &97	&5		&1.092965\\
\hline
\end{tabular}
\vskip .25 in
\end{table}
\begin{table}[tp]
 \caption{Least Prime $p$ and Maximal $k$-Tuple of Primitive Roots.} \label{tb9799.06}\centering%
\begin{tabular}{ l | l|l| l | l|l }
\hline
$p$	             &$k$   &$\hat{m}$         &$p$	               &$k$		&$\hat{m}$\\
\hline
$3=2\cdot 1+1$       &1     &0.910239226       &$83=2^2\cdot 41+1$	&7	        &1.584125933\\                                                   
$5=2\cdot 2+1$	     &2	    &1.242669869       &$347=2\cdot 173+1$	&8		&1.367679228\\  
$11=2\cdot 5+1$	     &3	    &1.251097174       &$269=2^2\cdot 67+1$	&9		&1.608662072\\
$37=2^2\cdot 3^2+1$    &4    &1.107751574       &$563=2\cdot 281+1$  &10		&1.578960758\\
$53=2^2\cdot 13+1$    &5	&1.259353244     &$467=2\cdot 233+1$        &11		&1.789686094\\
$89=2^3\cdot 11+1$   &6	    &1.336708859   &$1187=2^2\cdot 3^3\cdot 11+1$&12     &1.695110528\\
\hline
\end{tabular}
\vskip .25 in
\end{table}
\section{Consecutive Primitive Roots} \label{s9090}
Consecutive primitive roots is one of the simplest configuration of a subset of two or more primitive roots.
A more general result was proved by Carlitz \cite{CL56} using a counting technique based on Lemma \ref{lem333.02}. A new proof and counting technique 

based on Lemma \ref{lem333.03} is given here.
\subsection{Strings Of $k+1$ Consecutive Primitive Roots}  

Let  $ a_0, a_1, a_2, \ldots,a_k$ be a fixed $(k+1)$-tuple of distinct integers. Let $p \geq 2$ be a large prime, and let $\tau \in \F_p$ be a primitive root. A string of $k+1$ consecutive primitive roots $n+a_0, n+a_1, n+a_2, \ldots, n+a_k$ exists if and only if the system of equations 
\begin{equation} \label{eq9090.00}
\tau^{n_0}=n+a_0, \quad \tau^{n_1}=n+a_1, \quad\tau^{n_2}=n+a_2, \quad \ldots, \quad \tau^{n_k}=n+a_k,
\end{equation}
has one or more solutions. A solution consists of a $(k+1)$-tuple $n_0,n_1,\ldots, n_k$ of integers such that $\gcd(n_i,p-1)=1$ for $i=0,1, \ldots, k$, and some $n \in \F_p$. Let 

\begin{equation} \label{eq9090.02}
N(k,p)=\#\left \{ n \in \F_p: \ord_p (n+a_i)=p-1 \right \}
\end{equation}
for $i=0, 1, \ldots, k$, denotes the number of solutions. 

\begin{proof} {\normalfont (Theorem \ref{thm8800.040}): }The total number of solutions is written in terms of characteristic function for primitive roots, see Lemma \ref{lem333.03}, as
\begin{eqnarray} \label{eq9090.018}
N(k,p)&=&\sum_{ n\in \F_p}  \Psi \left(n+a_0\right)\Psi \left(n+a_1\right)\cdots \Psi \left(n+a_k\right) \\
&=&\sum_{ n\in \F_p}  \prod_{0 \leq i\leq k} \left (\frac{1}{p}\sum_{\substack{\gcd(n_i,p-1)=1\\
0\leq u_i\leq p-1}} \psi \left((\tau ^{n_i}-n-a_i)u_i)\right) \right )  \nonumber\\
&=&M(k,p) + E(k,p)\nonumber.
\end{eqnarray} 
The main term is determined by the indices $u_0=u_1=\cdots=u_k=0$, and has the form 
\begin{equation} \label{eq9090.020}
M(k,p)
=\sum_{ n\in \F_p}   \prod_{0 \leq i\leq k}\left (\frac{1}{p}\sum_{\gcd(n_i,p-1)=1} 1 \right )  ,
\end{equation} 
and the error term  is determined by the indices $u_0\ne0,u_1\ne0,\ldots,u_k\ne0$, and has the form 
\begin{equation} \label{eq9090.022}
E(k,p)=\sum_{ n\in \F_p}  \prod_{0 \leq i\leq k} \left (\frac{1}{p}\sum_{\substack{\gcd(n_i,p-1)=1\\
1\leq u_i\leq p-1}} \psi \left((\tau ^{n_i}-n-a_i)u_i)\right) \right )  .
\end{equation}

Applying Lemma \ref{lem9799.76} to the main term and Lemma \ref{lem8899.06} to the error term, yield
\begin{eqnarray} \label{eq9050.038}
N(k,p)
&=&M(k,p) + E(k,p)\\
&=& \left (\frac{\varphi(p-1)}{p-1} \right )^{k+1}p +O(\log^2 p)+O(p^{1-\varepsilon})\nonumber\\
&=&\left (\frac{\varphi(p-1)}{p-1} \right )^{k+1}p +O(p^{1-\varepsilon})\nonumber\\
&>&0 \nonumber,
\end{eqnarray} 
for all sufficiently large primes $p \geq 2$, and an arbitrary small number $\varepsilon>0$.

\end{proof}

\section{Probabilities Functions For Consecutive Primitive Roots}\label{s1188}
The forms of the main terms in Theorem \ref{thm8800.040} and Theorem \ref{thm8800.050} imply that a primitive root in a finite field $\F_p$ is a nearly independent random variable $X=X(p)$. 

\begin{dfn} \label{dfn1188.200} { \normalfont The probability of primitive roots in a finite field $\F_p$ is defined by
\begin{equation} \label{eq1188.055}
P\left ( \ord_p\left (X\right )=p-1\right)=\frac{\varphi(p-1)}{p-1}+O \left ( \frac{1}{p^{\varepsilon}}\right),
\end{equation}
where $\varepsilon>0$ is a small number. }
\end{dfn}

The occurrence of each primitive root is approximately an independent variable $X$ with probability $P(\ord_p X=p-1)=\varphi(p-1)/(p-1)$, as demonstrated in Definition \ref{dfn1188.200}. A random $(k+1)$-tuple of consecutive  primitive roots is denoted by
\begin{equation} \label{eq1188.059}
Z_k=\left (X_0, X_1, \ldots, X_k\right ),
\end{equation}
where each primitive root $X_i$ has order $\ord_p\left (X_i\right )=p-1$. The Fermat prime numbers $p=2^{2^m}+1$ and the Germain primes $p=2^aq+1$, where $a\geq 1$ and $q\geq 2$ is prime, have the simpler totients $p-1$, see Section \ref{s9292}, and descriptions of the probabilities functions of the $k+1$-tuples. The precise form for Germain primes is
\begin{equation} \label{eq1188.055}
P\left ( Z_k\right)=\left (\frac{\varphi(p-1)}{p-1}\right )^{k+1}= \left (\frac{1}{2}-\frac{1}{2q}\right )^{k+1},
\end{equation}
Table \ref{tb9799.04} demonstrates this well, almost all the listed cases have Germain primes; the exception could be an instance of the Strong Law of Small Numbers. On the other extreme are the collections of highly composite totients $p-1$. The precise form for primorial primes $p=2\cdot 3\cdot 5\cdots q+1$, where 

$q\geq3$ is prime, is
\begin{equation} \label{eq1188.055}
P\left ( Z_k\right)=\left (\frac{\varphi(p-1)}{p-1}\right )^{k+1}= \prod_{r\leq q}\left (1-\frac{1}{q}\right )^{k+1},
\end{equation}
where $r\leq q$ ranges over the primes. Some numerical data are displayed in Figure \ref{fe9799.23} and Figure \ref{fe9799.29}.
 
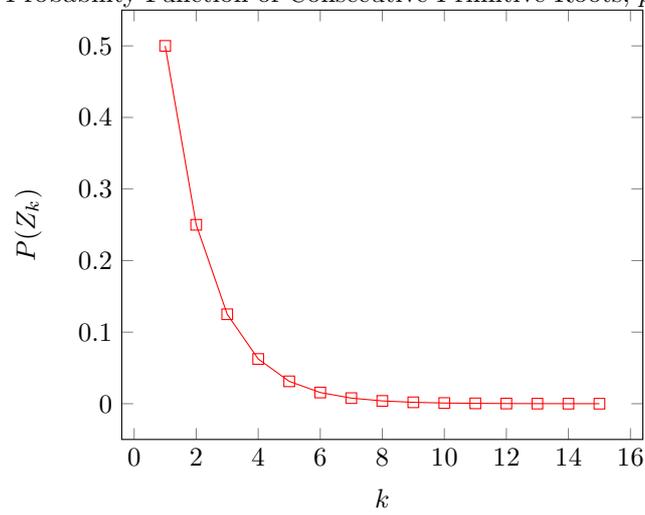
\begin{figure}[h]
\caption{Probability Function of Consecutive Primitive Roots, $p=2^{16}+1$} \label{fe9799.23}\centering%
  \begin{tikzpicture}
	\begin{axis}[
		xlabel=$k$,
		ylabel=$P(Z_k)$,
width=0.9\textwidth,
       height=0.5\textwidth		]
	\addplot[color=red,mark=square] coordinates {
		(1,1/2)
		(2,1/4)
		(3,1/8)
		(4,1/16)
		(5,1/32)
		(6,1/64)
		(7,1/128)
		(8,1/256)
		(9,1/2^9)
		(10,1/2^10)
		(11,1/2^11)
		(12,1/2^12)
		(13,1/2^13)
		(14,1/2^14)
		(15,1/2^15)
	};
	\end{axis}
\end{tikzpicture}	
\end{figure}

\begin{figure}[h!]
\caption{Probability Function of Consecutive Primitive Roots, $p=2\cdot3\cdots31+1$} \label{fe9799.29}\centering%
 \begin{tikzpicture}
	\begin{axis}[
		xlabel=$k$,
		ylabel=$P(Z_k)$,
	width=0.9\textwidth,
       height=0.5\textwidth	]
			\addplot[color=red,mark=square] coordinates {
		(1,.1529)
		(2,0.02337841)
		(3,.1529^3)
		(4,.1529^4)
		(5,.1529^5)
		(6,.1529^6)
		(7,.1529^7)
		(8,.1529^8)
		(9,.1529^9)
		(10,.1529^10)
		(11,.1529^11)
		(12,.1529^12)
		(13,.1529^13)
		(14,.1529^14)
		(15,.1529^15)
	};
\end{axis}
\end{tikzpicture}
\end{figure}

\section{Consecutive Squarefree Primitive Roots} \label{s9050}
The result for the existence of multiple consecutive squarefree primitive roots seems to be new in the literature. The first cases for 2 consecutive squarefree primitive roots $n$ and $n+1$, and 3 consecutive squarefree primitive roots $n$, $n+1$ and $n+2$ are feasible. But, the existence of 4 consecutive squarefree primitive roots $n$, $n+1$, $n+2$ and $n+3$ is infeasible. However, there are quasi consecutive squarefree primitive roots of length $k \ll \log p$ for a wide range of prime numbers. To describe these possibilities, let $(a_0, a_1, a_2, \ldots, a_k)$ be a fixed integers $(k+1)$-tuple of distinct integers. A string of $k+1$ quasi consecutive squarefree primitive roots $n+a_0, n+a_1, n+a_2, \ldots, n+a_k$ is a solution of the systems of equations: 
\begin{enumerate} [font=\normalfont, label=\arabic*.]
\item$ \displaystyle 
\tau^{n_0}=n+a_0, \quad \tau^{n_1}=n+a_1, \quad \tau^{n_2}=n+a_2, \quad \ldots, \quad \tau^{n_k}=n+a_k,$  \tabto{9cm} the primitive root condition.
\item$ \displaystyle 
\mu^{2}(n+a_0)=1, \quad \mu^{2}(n+a_1)=1,\quad \mu^{2}(n+a_2)=1,\quad \ldots, \quad \mu^{2}(n+a_k)=1,$ \tabto{9cm} the squarefree condition.
\end{enumerate}
A solution is a tuple $(n, n_0, n_1, \ldots , n_k) \in \N ^{k+2}$, with $\gcd(n_i,p-1)=1$, for $i=0,1, \ldots k$.  Let
\begin{equation} \label{eq9050.09}
N_2(k,p)=\# \left \{ n \in \F_p: \ord_p (n+a_i)=p-1, \mu^{2}(n+a_i)=1 \right \}
\end{equation}
for $i=0, 1, \ldots, k$, denotes the number of solutions.

\subsection{Strings Of $k+1$ Consecutive Squarefree Primitive Roots} \label{s8800}

 \begin{proof}{\normalfont (Theorem \ref{thm8800.050}): } The total number of solutions is written in terms of characteristic function for primitive roots, see Lemma \ref{lem333.03}, and the characteristic function for squarefree integers, see Lemma \ref{lem297.58}, as
\begin{eqnarray} \label{eq9050.018}
N_2(k,p)&=&\sum_{ n\in \F_p}   \prod_{0 \leq i\leq k} \Psi \left(n+a_i\right)\mu^{2}(n+a_i) 
\\
&=&\sum_{ n\in \F_p}    \prod_{0 \leq i\leq k}  \left (\frac{1}{p}\sum_{\substack{\gcd(n_i,p-1)=1\\
0\leq u_i\leq p-1}} \psi \left((\tau ^{n_i}-n-a_i)u_i)\right) \right )  \nonumber\\
&=&M_2(k,p) + E_2(k,p)\nonumber.
\end{eqnarray} 
The main term is determined by the indices $u_0=u_1=\cdots=u_k=0$, and has the form 
\begin{equation} \label{eq9050.020}
M_2(k,p)
=\sum_{ n\in \F_p}   \prod_{0 \leq i\leq k}\left (\frac{\mu^{2}(n+a_i)}{p}\sum_{\gcd(n_i,p-1)=1} 1 \right )  ,
\end{equation} 
and the error term  is determined by the indices $u_0\ne0,u_1\ne0,\ldots,u_k\ne0$, and has the form 
\begin{equation} \label{eq9050.022}
E_2(k,p)
=\sum_{ n\in \F_p}   \prod_{0 \leq i\leq k} \left (\frac{\mu^{2}(n+a_i)}{p}\sum_{\substack{\gcd(n_i,p-1)=1\\
1\leq u_i\leq p-1}} \psi \left((\tau ^{n_i}-n-a_i)u_i)\right) \right )  .
\end{equation}

Applying Lemma \ref{lem9709.76} to the main term and Lemma \ref{lem8829.08} to the error term, yield
\begin{eqnarray} \label{eq9050.028}
N_2(k,p)
&=&M_2(k,p) + E_2(k,p)\\
&=& \prod_{q\geq 2}\left ( 1-\frac{\omega(q)}{q^2}\right )\left (\frac{\varphi(p-1)}{p-1} \right )^{k+1}p+ O\left ( p^{2/3}\right )+O(p^{1-\varepsilon})\nonumber\\
&=& \prod_{q\geq 2}\left ( 1-\frac{\omega(q)}{q^2}\right )\left (\frac{\varphi(p-1)}{p-1} \right )^{k+1}p+O(p^{1-\varepsilon})\nonumber\\
&>&0 \nonumber,
\end{eqnarray} 
for all sufficiently large primes $p \geq 2$, and an arbitrary small number $\varepsilon>0$.
\end{proof}

\subsection{Squarefree Primitive Roots}

Let $p \geq 2$ be a large prime, and let $\tau \in \F_p$ be a primitive root. A squarefree primitive root $n\in \F_p$ exists if and only if the system of equations 
\begin{equation} \label{eq9050.90}
\tau^{m}=n \quad \text{ and } \quad \mu^{2}(n)=1,
\end{equation}
has one or more solutions $(m, n) \in \N \times \N$ such that $\gcd(m, p-1)=1$, and $n \geq 2$. Let
\begin{equation} \label{eq9010.09}
N_2( p)=\#\left \{ n \in \F_p: \ord_p (n)=p-1 \text{ and } \mu^{2}(n)=1 \right \}
\end{equation}
denotes the number of solutions. 

\begin{thm}  \label{thm9010.040}  For any large prime \(p\geq 2\), the finite field $\F_p$ contains squarefree primitive roots. Furthermore, the total number has the asymptotic formula
\begin{equation} \label{eq9010.045}
N_2(p)=\prod_{q \geq 2}\left (1-\frac{1}{q^2}\right )\left (\frac{\varphi(p-1)}{p} \right )p+O(p^{1-\varepsilon}) ,
\end{equation} 
where $\varepsilon>0$ is an arbitrary small number.
\end{thm}
\begin{proof}The total number of solutions is written in terms of characteristic function for primitive roots, see Lemma \ref{lem333.03}, and the characteristic function for squarefree integers, see Lemma \ref{lem297.58}, as
\begin{eqnarray} \label{eq9010.018}
\sum_{ n\in \F_p}  \Psi \left(n\right)\mu^{2}(n) 
&=&\sum_{ n\in \F_p}   \left (\frac{\mu^{2}(n)}{p}\sum_{\substack{\gcd(m,p-1)=1\\
0\leq u\leq p-1}} \psi \left((\tau ^{m}-n)u\right) \right )  \nonumber \\
&=&M_2(p) + E_2(p).
\end{eqnarray} 
The main term $M_2(p)$ is determined by the index $u=0$, and the error term $E_2(p)$ is determined by the index $u\ne0$. Applying Lemma \ref{lem9739.06} to the main term and Lemma \ref{lem8829.08} to the error term, yield
\begin{eqnarray} \label{eq9010.028}
N_2(p)
&=&M_2(p) + E_2(p)\\
&=& \prod_{q \geq 2}\left (1-\frac{1}{q^2}\right )\left (\frac{\varphi(p-1)}{p} \right )p+ O\left ( p^{1/2}\right )+O(p^{1-\varepsilon})\nonumber\\
&=&\prod_{q \geq 2}\left (1-\frac{1}{q^2}\right )\left (\frac{\varphi(p-1)}{p} \right )p+O(p^{1-\varepsilon})\nonumber\\
&>&0 \nonumber,
\end{eqnarray} 
for all sufficently large primes $p \geq 2$, and an arbitrary small number $\varepsilon>0$.
\end{proof}

\subsection{Squarefree Twin Primitive Roots}

Let $p \geq 2$ be a large prime, and let $\tau \in \F_p$ be a primitive root. Each squarefree twin primitive roots $n+a_0$ and $ n+a_1$ is a solution of the systems of equations
\begin{enumerate} [font=\normalfont, label=\arabic*.]
\item$ \displaystyle 
\tau^{n_0}=n+a_0, \quad \tau^{n_1}=n+a_1, $  \tabto{8cm} the primitive root condition.
\item$ \displaystyle 
\mu^{2}(n+a_0)=1, \quad \mu^{2}(n+a_1)=1,$ \tabto{8cm} the squarefree condition.
\end{enumerate}
A solution is a triple $(n,n_0,n_1)\in \N \times \N \times \N$ such that $\gcd(n_i,p-1)=1$ for $i=0,1$. Let 
\begin{equation} \label{eq9040.09}
N_2(2,p)=\#\left \{ n \in \F_p: \ord_p (n+a_i)=p-1, \text{ and } \mu^{2}(n+a_i)=1 \right \}
\end{equation}
for $i=0, 1$, denotes the number of solutions. 
\begin{thm}  \label{thm9040.040}  For any large prime \(p\geq 2\), the finite field $\F_p$ contains $2$ consecutive squarefree primitive roots. Furthermore, the number of pairs has the asymptotic formula
\begin{equation} \label{eq9040.045}
N_2(2,p)=\prod_{q \geq 2}\left (1-\frac{2}{q^2}\right )\left (\frac{\varphi(p-1)}{p-1} \right )^2p+O(p^{1-\varepsilon}) ,
\end{equation} 
where $\varepsilon>0$ is an arbitrary small number. The simplest case is for $a_0=0, a_1=1$.
\end{thm}
\begin{proof}The total number of solutions is written in terms of characteristic function for primitive roots, see Lemma \ref{lem333.03}, and the characteristic function for squarefree integers, see Lemma \ref{lem297.58}, as
\begin{eqnarray} \label{eq9040.018}
\N_2(2,p)
&=&sum_{ n\in \F_p}  \Psi \left(n+a_0\right)\Psi \left(n+a_1\right) \mu^{2}(n+a_0) \mu^{2}(n+a_1)\\
&=&\sum_{ n\in \F_p}   \left (\frac{\mu^{2}(n+a_0)}{p}\sum_{\substack{\gcd(n_0,p-1)=1\\
0\leq u_0\leq p-1}} \psi \left((\tau ^{n_0}-n-a_0)u_0\right) \right )  \nonumber \\
&& \times \left (\frac{\mu^{2}(n+a_1)}{p}\sum_{\substack{\gcd(n_1,p-1)=1\\
0\leq u_1\leq p-1}} \psi \left((\tau ^{n_1}-n-a_1)u_1\right) \right )\nonumber\\
&=&M_2(2,p) + E_2(2,p)\nonumber.
\end{eqnarray} 
The main term $M_2(2,p)$ is determined by the indices $u_0=u_1=0$, and the error term $E_2(2,p)$ is determined by the indices $u_0\ne0,u_1\ne0$. Applying Lemma \ref{lem9709.77} to the main term and Lemma \ref{lem8829.08} to the error term, yield
\begin{eqnarray} \label{eq9040.028}
N_2(2,p)
&=&M_2(2,p) + E_2(2,p)\\
&=& \prod_{q \geq 2}\left (1-\frac{2}{q^2}\right )\left (\frac{\varphi(p-1)}{p-1} \right )^2p+ O\left ( p^{2/3}\right )+O(p^{1-\varepsilon})\nonumber\\
&=& \prod_{q \geq 2}\left (1-\frac{2}{q^2}\right )\left (\frac{\varphi(p-1)}{p-1} \right )^2p+O(p^{1-\varepsilon})\nonumber\\
&>&0 \nonumber,
\end{eqnarray} 
for all sufficently large primes $p \geq 2$, and an arbitrary small number $\varepsilon>0$.
\end{proof}

\subsection{Squarefree Triple Primitive Roots}

Let $p \geq 2$ be a large prime, and let $\tau \in \F_p$ be a primitive root. Each squarefree triple primitive roots $n+a_0$, $ n+a_1$, and $ n+a_2$ is a solution of the systems of equations
\begin{enumerate} [font=\normalfont, label=\arabic*.]
\item$ \displaystyle 
\tau^{n_0}=n+a_0, \quad \tau^{n_1}=n+a_1, \quad \tau^{n_2}=n+a_2,$  \tabto{9cm} the primitive root condition.
\item$ \displaystyle 
\mu^{2}(n+a_0)=1, \quad \mu^{2}(n+a_1)=1,\quad \mu^{2}(n+a_2)=1,$ \tabto{9cm} the squarefree condition.
\end{enumerate}
A solution is a triple $(n,n_0,n_1,n_2)\in \N ^4$ such that $\gcd(n_i,p-1)=1$ for $i=0,1,2$. Let 
\begin{equation} \label{eq9030.09}
N_2(3,p)=\#\left \{ n \in \F_p: \ord_p (n+a_i)=p-1, \text{ and } \mu^{2}(n+a_i)=1 \right \}
\end{equation}
for $i=0, 1, 2$, denotes the number of solutions. 
\begin{thm}  \label{thm9030.050}  For any large prime \(p\geq 2\), the finite field $\F_p$ contains $3$ consecutive squarefree primitive roots. Furthermore, the number of pairs has the asymptotic formula
\begin{equation} \label{eq9030.045}
N_2(3,p)=\prod_{q \geq 2}\left (1-\frac{3}{q^2}\right )\left (\frac{\varphi(p-1)}{p-1} \right )^3p+O(p^{1-\varepsilon}) ,
\end{equation} 
where $\varepsilon>0$ is an arbitrary small number.
\end{thm}
\begin{proof}The simplest case is for $a_0=0, a_1=1,a_2=2$. The total number of solutions is written in terms of characteristic function for primitive roots, see Lemma \ref{lem333.03}, and the characteristic function for squarefree integers, see Lemma \ref{lem297.58}, as
\begin{eqnarray} \label{eq9030.018}
N_2(3,p)&=&\sum_{ n\in \F_p}  \Psi \left(n\right)\Psi \left(n+1\right) \Psi \left(n+2\right) \mu^{2}(n) \mu^{2}(n+1)\mu^{2}(n+2)\\
&=&\sum_{ n\in \F_p}   \left (\frac{\mu^{2}(n)}{p}\sum_{\substack{\gcd(n_0,p-1)=1\nonumber\\
0\leq u_0\leq p-1}} \psi \left((\tau ^{n_0}-n)u_0\right) \right )  \nonumber \\
&& \times \left (\frac{\mu^{2}(n+1)}{p}\sum_{\substack{\gcd(n_1,p-1)=1\\
0\leq u_1\leq p-1}} \psi \left((\tau ^{n_1}-n-1)u_1\right) \right )\nonumber\\
&& \times \left (\frac{\mu^{2}(n+1)}{p}\sum_{\substack{\gcd(n_2,p-1)=1\\
0\leq u_2\leq p-1}} \psi \left((\tau ^{n_2}-n-2)u_2\right) \right )\nonumber\\
&=&M_2(3,p) + E_2(3,p).
\end{eqnarray} 
The main term $M_2(3,p)$ is determined by the indices $u_0=u_1=u_2=0$, and the error term $E_2(3,p)$ is determined by the indices $u_0\ne0,u_1\ne0,u_2\ne0$. Applying Lemma \ref{lem9729.78} to the main term and Lemma \ref{lem8829.08} to the error term, yield
\begin{eqnarray} \label{eq9030.028}
N_2(3,p)
&=&M_2(3,p) + E_2(3,p)\\
&=&\prod_{q \geq 2}\left (1-\frac{3}{q^2}\right )\left (\frac{\varphi(p-1)}{p-1} \right )^3p+ O\left ( p^{2/3}\right )+O(p^{1-\varepsilon})\nonumber\\
&=& \prod_{q \geq 2}\left (1-\frac{3}{q^2}\right )\left (\frac{\varphi(p-1)}{p-1} \right )^3p+O(p^{1-\varepsilon})\nonumber\\
&>&0 \nonumber,
\end{eqnarray} 
for all sufficiently large primes $p \geq 2$, and an arbitrary small number $\varepsilon>0$.
\end{proof}

\section{Consecutive $s$-Power Free Primitive Roots}\label{s9110}

\subsection{$s$-Power Free Primitive Roots}
Let $p \geq 2$ be a large prime, and let $\tau \in \F_p$ be a primitive root. A $s$-power free primitive root $n\in \F_p$ exists if and only if the system of equations 
\begin{equation} \label{eq9110.90}
\tau^{m}=n \quad \text{ and } \quad \mu_s(n)=1,
\end{equation}
has one or more solutions $(m, n) \in \N \times \N$ such that $\gcd(m, p-1)=1$, and $n \geq 2$. Let  
\begin{equation} \label{eq9110.09}
N_s(p)=\#\left \{ n \in \F_p: \ord_p (n)=p-1, \mu_s(n)=\pm 1 \right \},
\end{equation}
see Lemma \ref{lem297.58}, denotes the number of solutions. 

\begin{thm}  \label{thm9110.040}  Let $s \geq 2$ be a fixed integer. For any large prime \(p\geq 2\), the finite field $\F_p$ contains squarefree primitive roots. Furthermore, the total number has the asymptotic formula
\begin{equation} \label{eq9110.045}
N_s(p)=\prod_{q \geq 2}\left (1-\frac{1}{q^s}\right )\left (\frac{\varphi(p-1)}{p-1} \right )p+O(p^{1-\varepsilon}) ,
\end{equation} 
where $\varepsilon>0$ is an arbitrary small number.
\end{thm}
\begin{proof} (Theorem \ref{thm8800.190}): The total number of solutions is written in terms of characteristic function for primitive roots, see Lemma \ref{lem333.03}, and the characteristic function for $s$-power free integers, see Lemma \ref{lem297.58}, as
\begin{eqnarray} \label{eq9110.018}
\sum_{ n\in \F_p}  \Psi \left(n\right)\mu_s(n) 
&=&\sum_{ n\in \F_p}   \left (\frac{\mu_s(n)}{p}\sum_{\substack{\gcd(n,p-1)=1\\
0\leq u\leq p-1}} \psi \left((\tau ^{n}-n)u\right) \right )  \nonumber \\
&=&M_s(p) + E_s(p).
\end{eqnarray} 
The main term $M_s(p)$ is determined by the indices $u=0$, and the error term $E_s(p)$ is determined by the indices $u\ne0$. Applying Lemma \ref{lem9739.06} to the main term and Lemma \ref{lem8829.08} to the error term, yield
\begin{eqnarray} \label{eq9110.028}
N_s( p)
&=&M_s(p) + E_s(p)\\
&=&\prod_{q \geq 2}\left (1-\frac{1}{q^s}\right )\left (\frac{\varphi(p-1)}{p-1} \right )p+ O\left ( p^{1/s}\right )+O(p^{1-\varepsilon})\nonumber\\
&=& \prod_{q \geq 2}\left (1-\frac{1}{q^s}\right )\left (\frac{\varphi(p-1)}{p-1} \right )p+O(p^{1-\varepsilon})\nonumber\\
&>&0 \nonumber,
\end{eqnarray} 
for all sufficiently large primes $p \geq 2$, and an arbitrary small number $\varepsilon>0$.
\end{proof}

\subsection{$s$-Power Free Twin Primitive Roots}

Given a triple of small integers $a_0\ne a_1 $ and $s \geq 2$. Let $p \geq 2$ be a large prime, and let $\tau \in \F_p$ be a primitive root. Each string of $2$ consecutive $s$-powerfree primitive roots $n+a_0$ and $ n+a_1$ is a solution of the systems of equations: 
\begin{enumerate} [font=\normalfont, label=\arabic*.]
\item$ \displaystyle 
\tau^{n_0}=n+a_0, \quad \tau^{n_1}=n+a_1;$  \tabto{8cm} the primitive root condition.
\item$ \displaystyle 
\mu_s(n+a_0)=1, \quad \mu_s(n+a_1)=1$;  \tabto{8cm}  the $s$-power free condition.
\end{enumerate}
A solution is a triple $(n, n_0, n_1) \in \N \times \N \times \N$, with $\gcd(n_i,p-1)=1$, for $i=0,1$. Let 
\begin{equation} \label{eq9110.09}
N_s(2,p,a)=\#\left \{ n \in \F_p: \ord_p (n+a_i)=p-1, \text{ and } \mu_s(n+a_i)=\pm 1 \right \},
\end{equation}
for $i=0,1$, denotes the number of solutions. 

\begin{proof} (Theorem \ref{thm8800.195}): The total number of solutions is written in terms of characteristic function for primitive roots, see Lemma \ref{lem333.03}, and the characteristic function for squarefree integers, see Lemma \ref{lem297.58}, as
\begin{eqnarray} \label{eq9110.018}
N_s(2,p)
&=&\sum_{ n\in \F_p}  \Psi \left(n+a_0\right)\Psi \left(n+a_1\right) \mu_s(n+a_0) \mu_s(n+a_1)\\
&=&\sum_{ n\in \F_p}   \left (\frac{\mu_s(n+a_0)}{p}\sum_{\substack{\gcd(n_0,p-1)=1\\
0\leq u_0\leq p-1}} \psi \left((\tau ^{n_0}-n-a_0)u_0\right) \right )  \nonumber \\
&& \times \left (\frac{\mu_s(n+a)}{p}\sum_{\substack{\gcd(n_1,p-1)=1\\
0\leq u_1\leq p-1}} \psi \left((\tau ^{n_1}-n-a_1)u_1\right) \right )\nonumber\\
&=&M_s(2,p) + E_s(2,p)\nonumber.
\end{eqnarray} 
The main term $M_s(2,p)$ is determined by the indices $u_0=u_1=0$, and the error term $E_s(2,p)$ is determined by the indices $u_0\ne0,u_1\ne0$. Applying Lemma \ref{lem9715.77} to the main term and Lemma \ref{lem8829.08} to the error term, yield
\begin{eqnarray} \label{eq9110.028}
N_s(2,p)
&=&M_s(2,p) + E_s(2,p)\\
&=& \prod_{q \geq 2}\left (1-\frac{\rho(s)}{q^s}\right )\left (\frac{\varphi(p-1)}{p-1} \right )^2p+ O\left (p^{\alpha(s)-\varepsilon}\right )+O(p^{1-\varepsilon})\nonumber\\
&=& \prod_{q \geq 2}\left (1-\frac{\rho(s)}{q^s}\right )\left (\frac{\varphi(p-1)}{p-1} \right )^2p+O(p^{1-\varepsilon})\nonumber\\
&>&0 \nonumber,
\end{eqnarray} 
where $\rho(s)=1,2$, and $\varepsilon>0$ is an arbitrary small number, for all sufficiently large primes $p \geq 2$.
\end{proof}

\section{Relatively Prime Primitive Roots } \label{s2887}
The first proof based on Lemma \ref{lem333.02} and restricted to $q=p-1$ was given in \cite{HM76}. A new proof based on Lemma \ref{lem333.03}, and for any $q\leq p-1$, is given here. The second result for consecutive and relatively prime to $q\geq2$ appears to be a new result in the literature.

\subsection{Relatively Prime Primitive Roots}
\begin{proof} (Theorem \ref{thm8800.090}) For a large prime $p\geq 2$, the total number of primitive roots relatively prime to a fixed integer $q$ is precisely
\begin{equation} \label{eq2887.40}
N_r(p,q)=\sum _{\substack{n\in \F_p\\
\gcd(n,q)=1}} \Psi (n).
\end{equation}
In terms of characteristic function for primitive roots, see Lemma \ref{lem333.03}, this is written as
\begin{eqnarray} \label{eq2887.50}
N_r(p,q)&=&\sum _{\substack{n\in \F_p\\
\gcd(n,q)=1}} \Psi (n) \nonumber \\
&=&\sum _{\substack{n\in \F_p\\
\gcd(n,q)=1}}   \left (\frac{1}{p}\sum_{\gcd(m,p-1)=1,} \sum_{ 0\leq u\leq p-1} \psi \left((\tau ^m-n)u\right) \right ) \\
&=& \frac{1}{p} \sum _{\substack{n\in \F_p\\
\gcd(n,q)=1}}  \sum_{\gcd(m,p-1)=1} 1+\frac{1}{p}\sum _{\substack{n\in \F_p\\
\gcd(n,q)=1}}
\sum_{\gcd(m,p-1)=1,} \sum_{ 0<u\leq p-1} \psi \left((\tau ^m-n)u\right)\nonumber\\
&=&M_r(p,q) + E_r(p,q)\nonumber.
\end{eqnarray} 
The main term $M_r(p,q)$ is determined by a finite sum over the trivial additive character $\psi =1$, and the error term $E_r(p,q)$ is determined by a finite sum over the nontrivial additive characters \(\psi(t) =e^{i 2\pi t  /p}\neq 1\). Applying Lemma \ref{lem9729.256} to the main term and Lemma \ref{lem8899.09} to the error term, yield
\begin{eqnarray} \label{eq2887.60}
N_r(p,q)
&=&M_r(p,q) + E_r(p,q)\\
&=&\frac{\varphi(q)}{q}\frac{\varphi(p-1)}{p-1}p+O(\log^2 p)+O(p^{1-\varepsilon}) \nonumber\\
&=&\frac{\varphi(q)}{q}\frac{\varphi(p-1)}{p-1}p+O(p^{1-\varepsilon}) \nonumber\\
&>&0 \nonumber,
\end{eqnarray} 
for all sufficiently large primes $p \geq 2$, and an arbitrary small number $\varepsilon>0$.
\end{proof}

\subsection{Relatively Prime Twin Primitive Roots}
The dependence correction factor $c_2(q,a)\geq 0$, and the parameter $q=q(a)$ depends on $a \geq 1$.  For instance, for $a=1$, the value $q=q(a)$ must be odd, and $c_2(q,a)>0$, otherwise $c_2(q,a)= 0$ for even $q$. Basically, the vanishing and nonvanishing are described in these cases:
\begin{equation}
c_2(q,a)
\left \{
\begin{array}{ll}
>0     &\text{  if } a=2b+1, \text{ and }q=2c+1, \text{ with } b\geq 0,c \geq 0, \\
=0           &\text{  if }a=2b+1, \text{ and }q=2c, \text{ with } b\geq 0,c \geq 1,\\
>0      &\text{  if }a=2b, \text{ and }q\geq 1, \text{ with } b \geq 1. \\
\end{array}
\right .
\end{equation}
To continue the analysis, assume that the parameters $a\geq 1$ and $q\geq 1$ are admissible, and $c_2(q,a)>0$. Let $p \geq 2$ be a large prime, and let $\tau \in \F_p$ be a primitive root. Each pair of quasi consecutive primitive roots  $n$, $ n+a$ and relatively prime to $q=q(a)\geq 2$ is a solution of the systems of equations: 
\begin{enumerate} [font=\normalfont, label=\arabic*.]
\item$ \displaystyle 
\tau^{n_0}=n, \quad \tau^{n_1}=n+a;$  \tabto{8cm} the primitive root condition.
\item$ \displaystyle 
\gcd(n,q)=1, \quad \gcd(n+a,q)=1.$ \tabto{8cm} the relatively prime condition.
\end{enumerate}
A solution is a triple $(n, n_0, n_1) \in \N \times \N \times \N$, with $\gcd(n_i,p-1)=1$, for $i=0,1$.  Let 
\begin{equation} \label{2887.132}
N_r(2,p,q)=\#\left \{ n \in \F_p: \ord_p (n)=\ord_p (n+a)=p-1, \gcd(n,q)= \gcd(n+1,q)=1 \right \}
\end{equation}
denotes the number of solutions. 

\begin{proof} (Theorem \ref{thm8800.095}): For a large prime $p\geq 2$, the total number of pairs of quasi consecutive primitive roots, both relatively prime to a fixed integer $q\geq 2$, is precisely
\begin{equation} \label{eq2887.40}
N_r(2,p,q)=\sum _{\substack{n\in \F_p\\
\gcd(n,q)=1\\
\gcd(n+a,q)=1}} \Psi (n)\Psi (n+a).
\end{equation}
In terms of characteristic function for primitive roots, see Lemma \ref{lem333.03}, this is written as
\begin{eqnarray} \label{eq2887.50}
N_r(2,p,q)&=&  \sum _{\substack{n\in \F_p\\
\gcd(n,q)=1\\
\gcd(n+a,q)=1}} \left (\frac{1}{p} \sum_{ \substack{0\leq u\leq p-1\\ \gcd(c,p-1)=1}} \psi \left((\tau ^c-n)u\right) \right )  \left (\frac{1}{p} \sum_{ \substack{0\leq v\leq p-1\\\gcd(d,p-1)=1}} \psi \left((\tau ^d-n-a)v\right) \right )\nonumber\\
&=&M_r(2,p,q) + E_r(2,p,q).
\end{eqnarray} 
The main term $M_r(2, p,q)$, which is determined by the indices $u=v=0$, has the form 
\begin{equation} \label{eq2887.140}
M_s(2,p,q)
=\sum _{\substack{n\in \F_p\\
\gcd(n,q)=1\\
\gcd(n+a,q)=1}} \left (\frac{1}{p} \sum_{ \substack{0\leq u\leq p-1\\\gcd(c,p-1)=1}}1 \right )  \left (\frac{1}{p} \sum_{ \substack{0\leq v\leq p-1\\\gcd(d,p-1)=1}} 1\right )  ,
\end{equation} 
and the error term $E_r(2,p,q)$, which is determined by the indices $u\ne0,v\ne0$, has the form 
\begin{equation} \label{eq2887.123}
E_r(2,p,q)
=\sum _{\substack{n\in \F_p\\
\gcd(n,q)=1\\
\gcd(n+a,q)=1}} \left (\frac{1}{p} \sum_{ \substack{1\leq u\leq p-1\\\gcd(c,p-1)=1}} \psi \left((\tau ^c-n)u\right) \right )  \left (\frac{1}{p} \sum_{ \substack{1\leq v\leq p-1\\\gcd(d,p-1)=1}} \psi \left((\tau ^d-n-a)v\right) \right ) .
\end{equation}

Applying Lemma \ref{lem9729.356} to the main term and Lemma \ref{lem8899.06} to the error term, yield

\begin{eqnarray} \label{eq2887.133}
N_r(2, p,q)
&=&M_r(2,p,q) + E_r(2, p,q)\\
&=&c_2(q,a)\left (\frac{\varphi(q)}{q}\right )^2\frac{\varphi(p-1)}{p-1}p+O(p^{\varepsilon})+O(p^{1-\varepsilon}) \nonumber\\
&=&c_2(q,a)\left (\frac{\varphi(q)}{q}\right )^2\frac{\varphi(p-1)}{p-1}p+O(p^{1-\varepsilon}) \nonumber\\
&>&0 \nonumber,
\end{eqnarray} 
where $c_2(q,a)>0$ is a dependence correction factor with respect to an admissible pair $a, q\geq 1$, for all sufficiently large primes $p \geq 2$, and an arbitrary small number $\varepsilon>0$.
\end{proof}

\section{Squarefree And Relatively Prime Primitive Roots} \label{s3387}
The first result for squarefree and relatively prime primitive roots with respect to a fixed integer $q\leq p-1$ is given here. The second result for squarefree and relatively prime twin primitive roots $n$, $n+a$ with respect to a fixed integer $q\geq2$, and conditional on Conjecture \ref{conj8009.105}, is a new result in the literature.

\subsection{Squarefree And Relatively Prime Primitive Roots }

\begin{thm}  \label{thm3387.011}  Let \(p\geq 2\) be a large prime, and let $q=O(\log p) $ be an integer. Then, the finite field $\F_p$ contains squarefree primitive roots relatively prime to $q\geq 2$. Furthermore, the number of such elements has the asymptotic formula
\begin{equation} \label{eq3387.045}
N_{sr}(p,q)=\frac{6}{\pi^2}\prod_{p\nmid q}\left ( 1+\frac{1}{p}\right )^{-1} \frac{\varphi(p-1)}{p-1}p +O(p^{1-\varepsilon}),
\end{equation}
where $\varepsilon>0$ is an arbitrary small number.
\end{thm}
\begin{proof}  For a large prime $p\geq 2$, the total number of primitive roots relatively prime to a fixed integer $q<p$ is precisely
\begin{equation} \label{eq3387.40}
N_{sr}(p,q)=\sum _{\substack{n\in \F_p\\
\gcd(n,q)=1}} \Psi (n) \mu(n)^2.
\end{equation}
In terms of characteristic function for primitive roots, see Lemma \ref{lem333.03}, this is written as
\begin{eqnarray} \label{eq3387.50}
N_r(p,q)&=&\sum _{\substack{n\in \F_p\\
\gcd(n,q)=1}} \Psi (n) \mu(n)^2 \\
&=&\sum _{\substack{n\in \F_p\\
\gcd(n,q)=1}}   \left (\frac{\mu(n)^2}{p}\sum_{\gcd(m,p-1)=1,} \sum_{ 0\leq u\leq p-1} \psi \left((\tau ^m-n)u\right) \right ) \nonumber\\
&=& \frac{1}{p} \sum _{\substack{n\in \F_p\\
\gcd(n,q)=1}}  \sum_{\gcd(m,p-1)=1} \mu(n)^2+\sum _{\substack{n\in \F_p\\
\gcd(n,q)=1}}\frac{\mu(n)^2}{p}
\sum_{\gcd(m,p-1)=1,} \sum_{ 0<u\leq p-1} \psi \left((\tau ^m-n)u\right)\nonumber\\
&=&M_{sr}(p,q) + E_{sr}(p,q) \nonumber.
\end{eqnarray} 
The main term $M_{sr}(p,q)$ is determined by a finite sum over the trivial additive character $\psi =1$, and the error term $E_{sr}(p,q)$ is determined by a finite sum over the nontrivial additive characters \(\psi(t) =e^{i 2\pi t  /p}\neq 1\). Applying Lemma \ref{lem9009.421} to the main term and Lemma \ref{lem8809.02} or Lemma \ref{lem8829.108} to the error term, yield
\begin{eqnarray} \label{eq3387.60}
N_{sr}(p,q)
&=&M_{sr}(p,q)  + E_{sr}(p,q) \\
&=&\frac{6}{\pi^2}\prod_{p\nmid q}\left ( 1+\frac{1}{p}\right )^{-1} \frac{\varphi(p-1)}{p-1}p +O(p^{1/2})+O(p^{1-\varepsilon}) \nonumber\\
&=&\frac{6}{\pi^2}\prod_{p\nmid q}\left ( 1+\frac{1}{p}\right )^{-1} \frac{\varphi(p-1)}{p-1}p +O(p^{1-\varepsilon}) \nonumber\\
&>&0 \nonumber,
\end{eqnarray} 
for all sufficiently large primes $p \geq 2$, and an arbitrary small number $\varepsilon>0$.
\end{proof}


\subsection{Squarefree And Relatively Prime Twin Primitive Roots}\label{ss2288}

The dependence correction factor $c_2(q,a)\geq 0$, and the parameter $q=q(a)$ depends on $a \geq 1$. Basically, the vanishing and nonvanishing are described in these cases:
\begin{equation}
c_2(q,a)= 
\left \{
\begin{array}{ll}
>0     &\text{  if } a=2b+1, \text{ and }q=2c+1, \text{ with } b\geq 0,c \geq 0, \\
=0           &\text{  if }a=2b+1, \text{ and }q=2c, \text{ with } b\geq 0,c \geq 1,\\
>0      &\text{  if }a=2b, \text{ and }q\geq 1, \text{ with } b \geq 1. \\
\end{array}
\right .
\end{equation}
To continue the analysis, assume that the parameters $a\geq 1$ and $q\geq 1$ are admissible, and $c_2(q,a)>0$. Let $p \geq 2$ be a large prime, and let $\tau \in \F_p$ be a primitive root. Each pair of squarefree twin primitive roots $n$, $ n+a$ and relatively prime to $q=q(a)\geq 2$ is a solution of the systems of equations: 
\begin{enumerate} [font=\normalfont, label=\arabic*.]
\item$ \displaystyle 
\tau^{n_0}=n, \quad \tau^{n_1}=n+a;$  \tabto{8cm} the primitive root condition.
\item$ \displaystyle 
\mu(n)^2, \quad \mu(n+a)^2$;  \tabto{8cm}  the squarefree condition.
\item$ \displaystyle 
\gcd(n,q)=1, \quad \gcd(n+a,q)=1.$ \tabto{8cm} the relatively prime condition.
\end{enumerate}
A solution is a triple $(n, n_0, n_1) \in \N \times \N \times \N$, with $\gcd(n_i,p-1)=1$, for $i=0,1$.  Let \begin{equation} \label{2887.132}
N_{sr}(2,p,q)=\#\left \{ n \in \F_p: \text{Conditions $1$, $2$, \text{and} $3$ \text{are satisfied.}} \right \}
\end{equation}
denotes the number of solutions.

\begin{thm}  \label{thm2288.077}  Assume Conjecture {\normalfont \ref{conj8009.105}}. Let \(p\geq 2\) be a large prime, let $a\geq 1$ and $q=O(\log p) $ be a pair of integers. Then, the finite field $\F_p$ contains a pair $n$ and $n+a$ of squarefree primitive roots and relatively prime to $q\geq 2$. Furthermore, the number of such pairs has the asymptotic formula
\begin{equation} \label{eq2288.045}
N_{sr}(2,p,q)=c_2(q,a)\prod_{p\nmid q}\left ( 1+\frac{1}{p}\right )^{-2}\prod_{p\geq 2}\left ( 1-\frac{2}{p^2}\right )\left (\frac{\varphi(p-1)}{p}\right )^2p+O(p^{1-\varepsilon}),
\end{equation}
where $c_2(q,a)>0$ is a dependence correction factor, and $\varepsilon>0$ is an arbitrary small number.
\end{thm}

\begin{proof} For a large prime $p\geq 2$, the total number of squarefree twin primitive roots, both relatively prime to a fixed integer $q\geq 2$, is precisely
\begin{equation} \label{eq2288.40}
N_{sr}(2,p,q)=\sum _{\substack{n\in \F_p\\
\gcd(n,q)=1\\
\gcd(n+a,q)=1}} \Psi (n)\Psi (n+a)\mu(n)^2\mu(n+a)^2.
\end{equation}
In terms of characteristic function for primitive roots, see Lemma \ref{lem333.03}, this is written as
\begin{eqnarray} \label{eq2288.50}
N_{sr}(2,p,q)
&=&  \sum _{\substack{n\in \F_p\\
\gcd(n,q)=1\\
\gcd(n+a,q)=1}} \left (\frac{\mu(n)^2}{p} \sum_{ \substack{0\leq u\leq p-1\\ \gcd(c,p-1)=1}} \psi \left((\tau ^c-n)u\right) \right )   \\
&&\times \left (\frac{\mu(n+a)^2}{p} \sum_{ \substack{0\leq v\leq p-1\\\gcd(d,p-1)=1}} \psi \left((\tau ^d-n-a)v\right) \right )\nonumber\\
&=&M_{sr}(2,p,q) + E_{sr}(2,p,q)\nonumber.
\end{eqnarray} 
The main term $M_{sr}(2,p,q)$ is determined by the indices $u=v=0$, and has the form \begin{equation} \label{eq2288.140}
M_s(2,p,q)
=\sum _{\substack{n\in \F_p\\
\gcd(n,q)=1\\
\gcd(n+a,q)=1}} \left (\frac{1}{p} \sum_{ \substack{0\leq u\leq p-1\\\gcd(c,p-1)=1}}\mu(n)^2 \right )  \left (\frac{1}{p} \sum_{ \substack{0\leq v\leq p-1\\\gcd(d,p-1)=1}} \mu(n+a)^2\right )  ,
\end{equation} 
and the error term $E_{sr}(2,p,q)$ is determined by the indices $u\ne0,v\ne0$, and has the form as \eqref{eq2288.50}. Applying Lemma \ref{lem9792.356} to the main term and Lemma \ref{lem8829.108} to the error term, yield
\begin{eqnarray} \label{eq2288.133}
N_{sr}(2, p,q)
&=&M_{sr}(2,p,q) + E_{sr}(2,p,q)\\
&=&c_2(q,a)\prod_{p\nmid q}\left ( 1+\frac{1}{p}\right )^{-2}\prod_{p\geq 2}\left ( 1-\frac{2}{p^2}\right )\left (\frac{\varphi(p-1)}{p}\right )^2p+O\left (p^{1-\delta} \right )+O(p^{1-\varepsilon}) \nonumber\\
&=&c_2(q,a)\prod_{p\nmid q}\left ( 1+\frac{1}{p}\right )^{-2}\prod_{p\geq 2}\left ( 1-\frac{2}{p^2}\right )\left (\frac{\varphi(p-1)}{p}\right )^2p+O(p^{1-\varepsilon}) \nonumber\\
&>&0 \nonumber,
\end{eqnarray} 
where $c_2(q,a)>0$ is an admissible dependence correction factor, for all sufficiently large primes $p \geq 2$, and an arbitrary small number $\varepsilon>0$.
\end{proof}

\section{Probabilities For Consecutive Squarefree Primitive Roots}\label{s1188}
The forms of the main terms in Theorem \ref{thm8800.080} and Theorem \ref{thm8800.190} imply that a squarefree primitive root in a finite field $\F_p$ is a nearly independent random variable $X=X(p)$.
\begin{dfn} \label{dfn1188.202} { \normalfont The probability of squarefree primitive roots in a finite field $\F_p$ is defined by
\begin{equation} \label{eq8800.055}
P\left ( \ord_p\left (X\right )=p-1 \text{ and } \mu(X)^2\ne 0\right)=\frac{\varphi(p-1)}{p-1}\prod_{q \geq 2}\left (1-\frac{1}{q^2}\right )+O \left ( \frac{1}{p^{\varepsilon}}\right),
\end{equation}
where $\varepsilon>0$ is a small number.}
\end{dfn}

Some calculations described below demonstrates that two or more consecutive squarefree primitive roots are dependent random variables. 
\begin{lem} \label{lem1188.06}  Let \(p\geq 2\) be a large prime. Let $X_i$ be a random squarefree primitive root. Then, a pair of random consecutive squarefree primitive roots $X_0, X_1$ in a finite field $\F_p$ is a dependent random variable. Specifically, the probability of a pair of random consecutive squarefree primitive roots is 
\begin{eqnarray} \label{eq1188.08}
&&P(\ord(X_0) =p-1, \ord(X_1) =p-1 \text{ and } \mu(X_0) =\pm 1, \mu(X_1) =\pm 1)\nonumber\\
&=&\left (\frac{\varphi(p-1)}{p-1} \right )^2\prod_{q \geq 2}\left (1-\frac{2}{q^2}\right )+O \left ( \frac{1}{p^{\varepsilon}}\right),
\end{eqnarray} 
where $\varepsilon>0$ is a small number. 
\end{lem}

\begin{proof} The density constant in the main term of Theorem \ref{thm9040.040} is the probability of having two consecutive squafree primitive roots. Next, use a series of steps to reduces to a simpler product:
\begin{eqnarray} \label{eq1188.020}
\left (\frac{\varphi(p-1)}{p-1} \right )^2\prod_{q\geq 2}\left (1-\frac{2}{q^2}\right )
&=&\left (\frac{\varphi(p-1)}{p-1} \right )^2\prod_{q \geq 2}\left (1-\frac{1}{q^2}\right )^2\left (1+\frac{1}{q^2(q^2-2)}\right )^{-1}\nonumber\\
&<&\left (\frac{\varphi(p-1)}{p-1} \right )^2\prod_{q \geq 2}\left (1-\frac{1}{q^2}\right )^2.
\end{eqnarray}
The last line is product of the individual probabilities, which implies that the two properties of the consecutive random integers $X_0, X_1$ are independent. The reduction from independent events is measured by the dependence correction factor
\begin{equation} \label{eq1188.024}
c_2(2)=\prod_{q \geq 2}\left (1+\frac{1}{q^2(q^2-2)}\right )^{-1}=0.87298595344931361877174511\ldots.
\end{equation}

\end{proof}

\begin{lem} \label{lem1188.08}  Let \(p\geq 2\) be a large prime. Let $X_i$ be a random squarefree primitive root. Then, a triple of random consecutive squarefree primitive roots $X_0, X_1, X_2$ in a finite field $\F_p$ is a dependent random variable. Specifically, the probability of a triple of random consecutive squarefree primitive roots is 
\begin{eqnarray} \label{eq1188.118}
&&P(\ord(X_0) = \ord(X_1) =\ord(X_2) =p-1 \text{ and } \mu(X_0) =\pm 1, \mu(X_1) =\pm 1, \mu(X_2) =\pm 1)\nonumber\\
&=&\left (\frac{\varphi(p-1)}{p-1} \right )^3\prod_{q \geq 2}\left (1-\frac{3}{q^2}\right )+O \left ( \frac{1}{p^{\varepsilon}}\right),
\end{eqnarray} 
where $\varepsilon>0$ is a small number. 
\end{lem}

\begin{proof} The density constant in the main term of Theorem \ref{thm9040.040} is the probability of having two consecutive squafree primitive roots. Next, use a series of steps to reduces to a simpler product:
\begin{eqnarray} \label{eq1188.120}
\left (\frac{\varphi(p-1)}{p-1} \right )^3\prod_{q\geq 2}\left (1-\frac{3}{q^2}\right )
&=&\left (\frac{\varphi(p-1)}{p-1} \right )^3\prod_{q \geq 2}\left (1-\frac{1}{q^2}\right )^3\left (1+\frac{3q^2-1}{q^4(q^2-3)}\right )^{-1}\nonumber\\
&<&\left (\frac{\varphi(p-1)}{p-1} \right )^3\prod_{q \geq 2}\left (1-\frac{1}{q^2}\right )^3.
\end{eqnarray}
The last line is product of the individual probabilities, which implies that the two properties of the consecutive random integers $X_0, X_1, X_2$ are independent. The reduction from independent events is measured by the dependence correction factor
\begin{equation} \label{eq1188.148}
c_2(3)=\prod_{q \geq 2}\left (1+\frac{3q^2-1}{q^4(q^2-3)}\right )^{-1}=0.558526979127689105533330\ldots.
\end{equation}

\end{proof}

The pattern of the probability function for consecutive squarefree primitive roots breaks down for 4 consecutive squarefree primitive roots since $\left (1-\frac{4}{q^2}\right )=0$ at $q=2$.
\newpage
\section{Problems}
Several interesting problems of different level of complexities are presented in this section. The range of difficulty ranges from easy to very difficult.

\subsection{Least Consecutive Primitive Roots In Finite Fields}
\begin{exe} { \normalfont 
Let $p \geq 2$ be a large prime, and let $k=2$. Determine an asymptotic formula for the least pair of consecutive primitive roots $n$ and $n+1$ in the finite field $\F_p$. Is the magnitude $n=O(\log^c p)$, where $c>0$ is a constant, correct?}
\end{exe}
\begin{exe} { \normalfont 
Let $p \geq 2$ be a large prime, and let $k=2$. Determine an asymptotic formula for the least pair of consecutive squarefree primitive roots $n$ and $n+1$ in the finite field $\F_p$. Is the magnitude $n=O(\log^c p)$, where $c>0$ is a constant, correct?}
\end{exe}
\begin{exe} { \normalfont 
Let $p \geq 2$ be a large prime, and let $k=3$. Determine an asymptotic formula for the least pair of consecutive primitive roots $n$, $n+1$, and $n+2$ in the finite field $\F_p$. Is the magnitude $n=O(\log^c p)$, where $c>0$ is a constant, correct?}
\end{exe}

\begin{exe} { \normalfont 
Let $p \geq 2$ be a large prime, and let $k=3$. Determine an asymptotic formula for the least pair of consecutive squarefree primitive roots $n$, $n+1$, and $n+2$ in the finite field $\F_p$. Is the magnitude $n=O(\log^c p)$, where $c>0$ is a constant, correct?}
\end{exe}

\begin{exe} { \normalfont 
Show that there are infinitely many admissible $4$-tuples $(a_0,a_1,a_2,a_3)$, and each one generates  infinitely many squarefree integers $4$-tuples $(n+a_0,n+a_1,n+a_2,n+a_3)$ as $n \to \infty$. For example, $(n,n+1,n+3,n+5)$, with $n \geq 1$.}
\end{exe}
\subsection{Simultaneous Primitive Root In Finite Fields}

\begin{exe} { \normalfont 
Let $p \geq 2$ and $q\geq 2$ be large distinct primes. Develop an algorithm for computing a simultaneous primitive root $u\ne\pm 1, v^2$ modulo $p$ and modulo $q$.}
\end{exe}

\begin{exe} { \normalfont 
Let $p \geq 2$, $q\geq 2$, and $r\geq 2$ be large distinct primes. Develop an algorithm for computing a simultaneous primitive root $u\ne\pm 1, v^2$ modulo $p$ modulo $q$, and $r$.}
\end{exe}

\subsection{Consecutive And Relatively Prime Primitive Roots }
\begin{exe} { \normalfont 
Let $p \geq 2$ be a large prime, and let $q \geq 1$ be a fixed integer. Prove that there are infinitely many consecutive  prime primitive roots and relatively prime to $q$. Determine an asymptotic formula for the number of $k\geq 3$ consecutive primitive roots $n$, $n+a$, and $n+b$ in the finite field $\F_p$ and relatively prime to $q$. }
\end{exe}

\begin{exe} { \normalfont 
Let $p \geq 2$ be a large prime, and let $q \geq 1$ be a fixed integer. Prove a result on the distribution of pairs of consecutive primitive roots relatively prime to $q$.  }
\end{exe}
\begin{exe} { \normalfont 
Let $p \geq 2$ be a large prime, and let $B \geq 1$ be a fixed integer. Prove the existence of pairs of consecutive smooth primitive roots relative to $B$.  }
\end{exe}

\subsection{Summatory Functions And Primitive Roots }
\begin{exe} { \normalfont 
Let $s \in \Z$ be a fixed integer, and let $p \geq 1$ be a prime. Evaluate the finite sum$$
\sum _{n<p} \Psi (n)n^s.
$$. }
\end{exe}

\begin{exe} { \normalfont 
Let $s \in \Z$ be a fixed integer, and let $p \geq 1$ be a prime. Evaluate the finite sum$$
\sum _{n<p} \Psi (n)n^s\mu(n).
$$. }
\end{exe}
\subsection{Length Merit Factor }
\begin{exe} { \normalfont 
Determine the an effective upper bound $C>0$ for the length merit factor $m=k/\log p\leq C$ for all primes $p \geq 2$, see Definition \ref{dfn9799.100}.
 }
\end{exe}
\begin{exe} { \normalfont 
Compute a table of the length merit factor $m=k/\log p$ indexed by the primes $p \leq 1000$.
. }
\end{exe}
\begin{exe} { \normalfont 
Compute a table of the length merit factor $m=k/\log p$ indexed by the length $k \leq 50$.
 }
\end{exe}

\currfilename.\\


\newpage
\begin{thebibliography}{999}
\bibitem{BH62} Bateman, P. T., Horn, R. A. \textit{\color{blue}A heuristic asymptotic formula concerning the distribution of prime numbers}, Math. Comp. 16 363-367, 1962.
\bibitem{BJ13} Julia Brandes. \textit{\color{blue}Twins of $s$-free numbers}. arXiv:1307.2066. 


\bibitem{BT13} William D. Banks, Tristan Freiberg, Caroline L. Turnage-Butterbaugh. \textit{\color{blue}Consecutive primes in tuples,} arXiv:1311.7003.
\bibitem{BW98} Berndt, Bruce C.; Evans, Ronald J.; Williams, Kenneth S. \textit{\color{blue}Gauss and Jacobi sums.} Canadian Math. Soc. Series of Monographs. A Wiley-Interscience Publication. New York, 1998. 
\bibitem{CC09} Cobeli, Cristian. \textit{\color{blue}On a Problem of Mordell with Primitive Roots}, arXiv:0911.2832.
\bibitem{CG02} Caldwell, Chris K.; Gallot, Yves. \textit{\color{blue}On the primality of $n!\pm 1$ and $2\times 3\times 5\times \cdots \times p\pm1$}. Math. Comp. 71 (2002), no. 237, 441-448. 
\bibitem{CL56} Carlitz, L. \textit{\color{blue}Sets of primitive roots.} Compositio Math. 13 (1956), 65-70. 
\bibitem{CL32} L. Carlitz. \textit{\color{blue}On a problem in additive arithmetic}, Quart. J. Math., 3, (1932), 273-290. 
\bibitem{CN18} N. A. Carella. \textit{\color{blue}Primitive Roots In Short Intervals}, arXiv:1806.01150.
 \bibitem{CP05} Crandall, Richard; Pomerance, Carl. \textit{\color{blue}Prime numbers. A computational perspective}. Second edition. Springer, New York, 2005.
 \bibitem{CS85} Cohen, Stephen D. \textit{\color{blue}Consecutive primitive roots in a finite field}. Proc. Amer. Math. Soc. 93 (1985), no. 2, 189-197. 
\bibitem{CZ98} Cobeli, Cristian; Zaharescu, Alexandru. \textit{\color{blue}On the distribution of primitive roots mod p }. Acta Arith. 83, (1998), no. 2, 143-153.

\bibitem{DH37} H. Davenport. \textit{\color{blue}On Primitive Roots in Finite Fields}, Quarterly J. Math. 1937, 308-312. 
\bibitem{DL12} De Koninck, Jean-Marie; Luca, Florian. \textit{\color{blue}Analytic number theory. Exploring the anatomy of integers}. Graduate Studies in Mathematics, 134. American Mathematical Society, Providence, RI, 2012.
\bibitem{DK05} De Koninck, J. M.; Katai, I. \textit{\color{blue}On the mean value of the index of composition of an integer}. Monatsh. Math.  145  (2005),  no. 2, 131-144.
\bibitem{DS12} Rainer Dietmann, Christian Elsholtz, Igor E. Shparlinski. \textit{\color{blue}On Gaps Between Primitive Roots in the Hamming Metric}, arXiv:1207.0842. 

\bibitem{ES57} Paul Erdos, Harold N. Shapiro. \textit{\color{blue}On The Least Primitive Root Of A Prime}, 1957, euclidproject.org.


\bibitem{FS02} Friedlander, John B.; Konyagin, Sergei; Shparlinski, Igor E. \textit{\color{blue}Some doubly exponential sums over $\Z_m$.} Acta Arith. 105, (2002), no. 4, 349-370.

 \bibitem{FS01}Friedlander, John B.; Shparlinski, Igor E. \textit{\color{blue}Double exponential sums over thin sets}. Proc. Amer. Math. Soc.  129  (2001),  no. 6, 1617-1621.	 
\bibitem{FS00} Friedlander, John B.; Hansen, Jan; Shparlinski, Igor E. \textit{\color{blue}Character sums with exponential functions}. Mathematika  47  (2000),  no. 1-2, 75-85 (2002).


	
 \bibitem{GR88} Guy, Richard K. \textit{\color{blue}The Strong Law of Small Numbers.} American Mathematical Monthly. 95 (8): 697-712, 1988. 
\bibitem{GZ05} Garaev, M. Z. \textit{\color{blue}Double exponential sums related to Diffie-Hellman distributions}. Int. Math. Res. Not.  2005,  no. 17, 1005-1014. 

\bibitem{GK05}  Garaev, M. Z.  A. A. Karatsuba. \textit{\color{blue}New estimates of double trigonometric sums with exponential functions,} arXiv:math/0504026.
\bibitem{GT07} Gun, S.; Luca, Florian; Rath, P.; Sahu, B.; Thangadurai, R. \textit{\color{blue}Distribution of residues modulo p} . Acta Arith.  129  (2007),  no. 4, 325-333. 
\bibitem{HR82} Hall, R. R. \textit{\color{blue}Squarefree numbers on short intervals}. Mathematika  29  (1982), no. 1, 7-17. 	
\bibitem{HC75} Hooley, C. \textit{\color{blue}A note on square-free numbers in arithmetic progressions}, Bull. Lond. Math. Soc. 7 (1975), 133-138.
\bibitem{HR82} Hall, R. R. \textit{\color{blue}Squarefree numbers on short intervals}. Mathematika  29  (1982), no. 1, 7-17. 
\bibitem{HM76} Hausman, Miriam. \textit{\color{blue}Primitive roots satisfying a co-prime condition}. Amer. Math. Monthly  83  (1976),  no. 9, 720-723. 
\bibitem{HW79}Hardy,  G. H.; Wright, E. M..  \textit{\color{blue}An Introduction to the Theory of Numbers}. 5th ed., Oxford University Press, Oxford, 1979. 
\bibitem{IK04} Iwaniec, Henryk; Kowalski, Emmanuel. \textit{\color{blue}Analytic number theory}. AMS Colloquium Publications,
53. American Mathematical Society, Providence, RI, 2004.
\bibitem{IA03} Ivic, Aleksandar. \textit{\color{blue}The Riemann zeta-function. Theory and applications}. Wiley, New York; Dover Publications, Inc., Mineola, NY, 2003.

\bibitem{KS12} Konyagin, Sergei V.; Shparlinski, Igor E. \textit{\color{blue}On the consecutive powers of a primitive root: gaps and exponential sums.} Mathematika  58  (2012),  no. 1, 11-20.
\bibitem{LE78} Lucas, Edouard. \textit{\color{blue}Theorie des Fonctions Numeriques Simplement Periodiques}. (French) Amer. J. Math.  1  (1878),  no. 4, 289-321. 
\bibitem{LZ07}Languasco, A.; Zaccagnini, A. \textit{\color{blue}A note on Mertens' formula for arithmetic progressions.} J. Number Theory  127  (2007),  no. 1, 37-46. 


\bibitem{LN97} Lidl, Rudolf; Niederreiter, Harald. \textit{\color{blue}Finite fields}. Second edition. Encyclopedia of Mathematics and its Applications, 20. Cambridge University Press, Cambridge, 1997.	

\bibitem{LT08}Luca, Florian; Shparlinski, Igor E.; Thangadurai, R. \textit{\color{blue}Quadratic non-residues versus primitive roots modulo p }. J. Ramanujan Math. Soc.  23  (2008),  no. 1, 97-104. 
\bibitem{MI17} Ingela Mennema. \textit{\color{blue}The distribution of consecutive square-free numbers}, Master Thesis, Leiden University, 2017.

\bibitem{ML47}  Mirsky, L. \textit{\color{blue}Note on an asymptotic formula connected with $r$-free integers}. Quart. J. Math., Oxford Ser.  18,  (1947). 178-182.
\bibitem{MP04} Moree, Pieter.\textit{\color{blue}Artin's primitive root conjecture -a survey}. arXiv:math/0412262. 

\bibitem{MP07} Moree, P.  \textit{\color{blue}Artin prime producing quadratics.} Abh. Math. Sem. Univ. Hamburg  77  (2007), 109-127.

\bibitem{ML72} Mordell, L. J. \textit{\color{blue}On the exponential sum $\sum_{1\leq x \leq X} exp (2\pi i(ax+bg^x )/p)$.} Mathematika  19  (1972), 84-87.
\bibitem{MV07} Montgomery, Hugh L.; Vaughan, Robert C. \textit{\color{blue}Multiplicative number theory. I. Classical theory}. Cambridge University Press, Cambridge, 2007.

\bibitem{NR14} Ramon M. Nunes. \textit{\color{blue}Square-free numbers in arithmetic progressions,} arXiv:1402.0684.
\bibitem{NW00} Narkiewicz, W. \textit{\color{blue}The development of prime number theory. From Euclid to Hardy and Littlewood.} Springer Monographs in Mathematics. Springer-Verlag, Berlin, 2000. 
\bibitem{OM14} Overholt, Marius.  \textit{\color{blue}A course in analytic number theory}. Graduate Studies in Mathematics, 160. American Mathematical Society, Providence, RI, 2014.  

\bibitem{PJ09} Pintz, Janos. \textit{\color{blue}Landau's problems on primes.} J. Theory. Nombres Bordeaux 21 (2009), no. 2, 357-404.

\bibitem{PS95}  Pappalardi, Francesco; Shparlinski, Igor. \textit{\color{blue}On Artin's conjecture over function fields.} Finite Fields Appl.  1  (1995),  no. 4, 399-404. 
\bibitem{PF05} Pappalardi, Francesco. \textit{\color{blue}A survey on $k$-freeness}. Number theory, 71-88, Ramanujan Math. Soc. Lect. Notes Ser., 1, Ramanujan Math. Soc., Mysore, 2005.
\bibitem{OR71} Orr, Richard C. \textit{\color{blue}Remainder estimates for squarefree integers in arithmetic progression}.  J. Number Theory  3  1971 474-497.

\bibitem{RI15} Igor Rivin. \textit{\color{blue}Some experiments on Bateman-Horn},  http://arxiv.org/abs/1508.07821.
\bibitem{RP96} Ribenboim, Paulo. \textit{\color{blue}The new book of prime number records}. Berlin, New York: Springer-Verlag, 1996.
\bibitem{RT12} T. Reuss. \textit{\color{blue}Pairs of k-free Numbers, consecutive square-full Numbers}. arXiv:1212.3150. 
	
\bibitem{RS62} J.B. Rosser and L. Schoenfeld. \textit{\color{blue}Approximate formulas for some functions of prime numbers}, Illinois J. Math. 6 (1962) 64-94.
\bibitem{RZ00} Rudnick, Zeev; Zaharescu, Alexandru. \textit{\color{blue}The distribution of spacings between small powers of a primitive root}. Israel J. Math.  120  (2000),  part A, 271-287. 


\bibitem{SM75} Szalay, Michael. \textit{\color{blue}On the distribution of the primitive roots of a prime.} J. Number Theory  7  (1975), 184-188.

	
\bibitem{SR73} Stoneham, R. G. \textit{\color{blue}On the uniform e-distribution of residues within the periods of rational fractions with applications to normal numbers}. Acta Arith.  22  (1973), 371-389.
\bibitem{SV92} Shoup, Victor. \textit{\color{blue}Searching for primitive roots in finite fields}. Math. Comp. 58 (1992), no. 197, 369-380.
\bibitem{SV08} Shoup, Victor. \textit{\color{blue}A computational introduction to number theory and algebra}. Cambridge University Press, Cambridge, 2005.
\bibitem{TK85} Tsang, Kai Man. \textit{\color{blue}The distribution of $r$-tuples of squarefree numbers}. Mathematika  32  (1985),  no. 2, 265-275 (1986). 
\bibitem{TT13} Tanti, Jagmohan; Thangadurai, R. \textit{\color{blue}Distribution of residues and primitive roots}. Proc. Indian Acad. Sci. Math. Sci.  123  (2013),  no. 2, 203-211. 
\bibitem{VE68}Vegh, Emanuel. \textit{\color{blue}Pairs of consecutive primitive roots modulo a prime}. Proc. Amer. Math. Soc.  19  (1968), 1169-1170.
\bibitem{VE69}Vegh, Emanuel. \textit{\color{blue} Primitive roots modulo a prime as consecutive terms of an arithmetic progression.} J. Reine Angew. Math. 235 (1969), 185-188. 
\bibitem{VR73} Vaughan, R. C. Some applications of Montgomery's sieve. J. Number Theory 5 (1973), 64-79.
\bibitem{WR80} R. Warlimont. \textit{\color{blue}Squarefree numbers in arithmetic progressions}, J. London Math. Soc. (2)22(1980), 21-24.
\bibitem{WR01} Winterhof, Arne. \textit{\color{blue}Character sums, primitive elements, and powers in finite fields}. J. Number Theory 91, 2001, no. 1, 153-163.





\end{thebibliography}
\end{document}